\documentclass[11pt]{amsart}
\usepackage[utf8]{inputenc}

\usepackage[abbrev,lite,alphabetic]{amsrefs}
\usepackage[colorlinks,plainpages,urlcolor=blue,linkcolor=reference,citecolor=citation,colorlinks = true, hyperfootnotes=false]{hyperref}
\hypersetup{citecolor=gray, linkcolor = darkgray, urlcolor= gray}

\usepackage[T1]{fontenc}
\usepackage{etex}
\usepackage[shortlabels]{enumitem}
\usepackage{moreenum}
\usepackage{subfiles}
\usepackage{lmodern}
\usepackage{wrapfig}
\usepackage{csquotes}
\usepackage{comment}
\usepackage{multicol}

\usepackage{tikz-cd}
\usepackage{mathrsfs}
\usepackage{sseq}
\usepackage{mathdots}
\usepackage{thm-restate}
\usepackage{amsmath}
\usepackage{amsxtra}
\usepackage{amscd}
\usepackage{amsthm}
\usepackage{amsfonts}
\usepackage{amssymb}
\usepackage{mathtools}
\usepackage[scr=rsfs]{mathalfa}
\usepackage{eucal}
\usepackage{bbm}
\usepackage[all,cmtip]{xy}

\usepackage{graphicx}
\usepackage{tikz}
\usepackage{morefloats}
\usepackage{pdfpages}
\usetikzlibrary{matrix,arrows,decorations.pathmorphing,quotes,angles,decorations.markings,shapes.misc}
\RequirePackage{color}

\usepackage[colorinlistoftodos, textsize=tiny]{todonotes}
\usepackage[margin=1in,marginparwidth=0.8in, marginparsep=0.1in]{geometry}
\usepackage[framemethod=TikZ]{mdframed}

\usepackage{cleveref}

\definecolor{reference}{rgb}{0.20,0.36,0.74}
\definecolor{citation}{rgb}{0,.40,.80}

\usepackage{lipsum}
\usepackage{cleveref}
\crefname{section}{\S \!\!}{\S\S \!\!}
\crefname{equation}{}{}
\crefname{enumi}{}{}
\crefname{appendix}{\S \!\!}{\S\S \!\!}  

\setenumerate[1]{label={(\arabic*)}}

\allowdisplaybreaks

\theoremstyle{plain}
\newtheorem{theorem}{Theorem}[subsection]

\newtheorem{proposition}[theorem]{Proposition}
\newtheorem{lemma}[theorem]{Lemma}
\newtheorem{corollary}[theorem]{Corollary}
\newtheorem*{corollary*}{Corollary}

\newtheorem*{conjecture*}{Conjecture}
\theoremstyle{definition}

\newtheorem{definition}[theorem]{Definition}

\newtheorem{notation}[theorem]{Notation}

\newtheorem{remark}[theorem]{Remark}
\newtheorem{example}[theorem]{Example}

\newtheorem*{example*}{Example}
\newtheorem*{remark*}{Remark}

\def\cA{\mathcal A}\def\cB{\mathcal B}\def\cC{\mathcal C}\def\cD{\mathcal D}
\def\cF{\mathcal F}\def\cH{\mathcal H}
\def\cI{\mathcal I}\def\cJ{\mathcal J}\def\cL{\mathcal L}
\def\cM{\mathcal M}\def\cN{\mathcal N}\def\cO{\mathcal O}\def\cP{\mathcal P}
\def\cT{\mathcal T}
\def\cU{\mathcal U}\def\cV{\mathcal V}\def\cW{\mathcal W}\def\cX{\mathcal X}

\newcommand{\fX}{\mathfrak{X}}

\newcommand{\fg}{\mathfrak{g}}
\newcommand{\fgv}{\mathfrak{g}^\vee}
\newcommand{\frf}{\mathfrak{f}}
\newcommand{\ffv}{\mathfrak{f}^\vee}
\newcommand{\fd}{\mathfrak{d}}
\newcommand{\fdv}{\mathfrak{d}^\vee}

\makeatletter
\providecommand{\leftsquigarrow}{%
  \mathrel{\mathpalette\reflect@squig\relax}%
}
\newcommand{\reflect@squig}[2]{%
  \reflectbox{$\m@th#1\rightsquigarrow$}%
}
\makeatother

\newcommand{\Shv}{{\sf SHV}}

\newcommand{\adj}{\dashv}

\newcommand{\Fun}{{\sf Fun}}
\newcommand{\Hom}{{\sf Hom}}

\newcommand{\Cat}{{\sf Cat}}
\newcommand{\bCat}{\widehat{\sf Cat}}

\renewcommand{\Pr}{{\sf Pr}}

\newcommand{\PrL}{\Pr^{\sf L}}

\newcommand{\PrLRomega}{\Pr^{\sf LR}_{\omega}}

\newcommand{\PrRomega}{\Pr^{R,\omega}}

\newcommand{\Spec}{{\sf Spec}}

\newcommand{\Mod}{{\sf Mod}}
\newcommand{\Modfd}{{\sf Mod}^{fd}}
\newcommand{\Perf}{{\sf Perf}}
\newcommand{\Prop}{{\sf Prop}}
\newcommand{\LPerf}{{\sf LPerf}}
\newcommand{\RPerf}{{\sf RPerf}}
\newcommand{\Qcoh}{{\sf QCoh}}

\newcommand{\LMod}{{\sf LMod}}
\newcommand{\RMod}{{\sf RMod}}

\renewcommand{\lim}{{\sf lim}}

\newcommand{\id}{{\sf id}}

\newcommand{\op}{{\sf op}}

\newcommand{\mon}{\sf{-mon}}

\renewcommand{\hom}{{\sf hom}}

\newcommand{\colim}{{\sf colim}}

\newcommand{\CC}{\mathbb{C}}

\newcommand{\EE}{\mathbb{E}}

\newcommand{\HH}{\mathbb{H}}

\newcommand{\LL}{\mathbb{L}}

\newcommand{\PP}{\mathbb{P}}
\newcommand{\RR}{\mathbb{R}}
\renewcommand{\SS}{\mathbb{S}}

\newcommand{\ZZ}{\mathbb{Z}}

\newcommand{\kk}{\Bbbk}

\newcommand{\Aut}{\mathsf{aut}}
\newcommand{\End}{\mathsf{end}}

\newcommand{\Athree}{{\cA_{\fX}}}
\newcommand{\Bthree}{{\cB_{\fX^\vee}}}

\newcommand{\QCoh}{\mathsf{QCoh}}

\newcommand{\IndCoh}{\mathsf{IndCoh}}

\newcommand{\Perv}{\mathsf{Perv}}

\newcommand{\Sh}{\mathsf{Sh}}
\newcommand{\sh}{\Sh}

\newcommand{\Cone}{\mathsf{Cone}}

\newcommand{\oA}{\overline{A}}

\newcommand{\ul}{\underline}
\newcommand{\muSh}{\mu\mathsf{Sh}}

\newcommand{\Loc}{\mathsf{Loc}}

\newcommand{\redu}{/\!\!/}

\newcommand{\Vect}{{\sf Vect}}

\newcommand{\heart}{\heartsuit}

\renewcommand{\Im}{\operatorname{Im}}

\newcommand{\FI}{t}
\newcommand{\mass}{m}

\newcommand{\dR}{{\sf dR}}
\newcommand{\Dmod}{{\sf Dmod}}
\newcommand{\Oskel}{\mathbb{L}_G(\FI,\mass)}
\newcommand{\Lstab}{\mathbb{L}_G(\FI,-)}
\newcommand{\Stab}{\Lstab}

\newcommand{\Bet}{{\sf Bet}}
\newcommand{\Betti}{{\sf Bet}}

\newcommand{\htvar}{\mathfrak{X}}
\newcommand{\htbar}{\overline{\mathfrak{X}}}

\newcommand{\bP}{\mathbf{P}}
\newcommand{\faces}{\mathbf{P}(\FI)}

\newcommand{\sProp}{{\sf Prop\,}}
\newcommand{\sPerf}{{\sf Perf\,}}
\newcommand{\sMod}{{\sf Mod\,}}
\newcommand{\Endo}{{\sf End}}
\newcommand{\fin}{{\sf fin}}
\newcommand{\univ}{{\sf univ}}
\newcommand{\FT}{{\sf FT}}

\newcommand{\RH}{RH}

\newcommand{\frF}{{\mathfrak{F}}}

\newcommand{\oC}{\overline{C}}
\newcommand{\tilG}{{\widetilde{G}}}
\newcommand{\tilU}{{\widetilde{U}}}
\newcommand{\tilLambda}{{\widetilde{\Lambda}}}
\renewcommand{\Re}{\operatorname{Re}}
\renewcommand{\Im}{\operatorname{Im}}
\newcommand{\pol}{\tau}
\newcommand*{\isoarrow}[1]{\arrow[#1,"\rotatebox{90}{\(\sim\)}"]}
\newcommand{\hpspace}{{\ffv_\RR(\FI)}}
\newcommand{\htpol}{\pol_G}
\newcommand{\mush}{\muSh}
\newcommand{\HP}{{\sf HP}}
\newcommand{\Core}{S_{{\sf skel}}}
\newcommand{\Cocore}{P_{{\sf skel}}}
\newcommand{\ocW}{\overline{\cW}}

\newcommand{\im}{\operatorname{im}}
\newcommand{\open}{\cL}
\renewcommand{\cong}{\simeq}

\newcommand{\eu}{\mathbf{eu}}

\title{Hypertoric Fukaya categories and categories $\mathcal{O}$}
\author{Laurent C\^ot\'e, Benjamin Gammage, and Justin Hilburn}

\begin{document}

\maketitle

\begin{abstract}
To a conical symplectic resolution with Hamiltonian torus action, Braden--Proudfoot--Licata--Webster associate a category $\cO,$ defined using deformation quantization (DQ) modules. 
It has long been expected, though not stated precisely in the literature, that category $\cO$ also admits a ``Betti-type'' realization as the Fukaya--Seidel category of a Lefschetz fibration.

In this paper, we confirm that the category $\cO$ associated to a toric hyperk\"ahler manifold is equivalent to the partially wrapped Fukaya category of a Liouville manifold stopped by the fiber of a $J$-holomorphic moment map.
The proof involves
relating earlier DQ-module computations to a new computation of microlocal perverse sheaves. 
Leveraging known results on (de Rham) hypertoric category $\cO,$ we deduce several Floer-theoretic consequences, including formality of simple objects and Koszul duality for the (fully) wrapped Fukaya category;
conversely, by applying results about microlocal sheaves, we produce a relative Calabi-Yau structure on category $\cO.$
\end{abstract}

\setcounter{tocdepth}{1}
\tableofcontents

\section{Introduction}\label{sec:intro}



\subsection{Category $\mathcal{O}$}

Let $\mathfrak{g}$ be a semisimple Lie algebra over $\mathbb{C}$. The classical {BGG category $\cO$} \cite{BGG} is a certain abelian subcategory of the category of all $\mathfrak{g}$-representations which is of great interest in representation theory. It is big enough to contain the highest weight modules and Verma modules but it is also small enough to enjoy many wonderful categorical properties \cite{humphreys2008representations}.  

Category $\cO$ also admits a purely geometric description, from which can be derived many of its useful properties.
Namely, by Beilinson-Bernstein localization \cite{BB-localization}, category $\cO$ arises as the category of $\cD$-modules on the flag variety $G/B$ constructible with respect to the Schubert stratification.
By the Riemann--Hilbert correspondence, category $\cO$ can equivalently be described as the category of Schubert-constructible perverse sheaves on $G/B$. This perspective opens category $\cO$ to the use of powerful tools from mixed Hodge theory, and it has been understood since \cite{BGS} that mixed geometry explains many properties of the Koszul algebras governing categories $\cO$.


Flag varieties (or rather their cotangent bundles) are the paradigmatic example of a somewhat loosely defined class of holomorphic symplectic manifolds called \emph{symplectic resolutions}. An important development in geometric representation theory over the past twenty years has been to generalize the tight relationship between the representation theory of category $\cO$ and the geometry of flag varieties to other symplectic resolutions. An early generalization of category $\cO$ appeared for Cherednik algebras in \cites{DO03,BEG03,GGER03}; the hypertoric category $\cO$, of interest in this paper, appeared in \cites{BLPW10} with geometric interpretation given in \cite{BLPW12}, followed by a general geometric construction of categories $\cO$ for conical symplectic resolutions in \cites{BLPW16}, with explicit computations appearing in \cite{Web-gencat-O}.


For the purpose of this introduction, we shall only consider symplectic resolutions which arise from hyperkähler reduction. Let $G$ be a reductive complex Lie group acting on a vector space $E$. The induced action on $T^*E$ is Hamiltonian and comes with a complex moment map $\mu^\mathbb{C}: T^*E \to \mathfrak{g}^\vee$. As reviewed in \Cref{subsection:hypertorics}, one can form the hyperkähler reduction $ \htvar_G(\FI) := (\mu^{\mathbb{C}})^{-1}(0)\redu_\FI G$ at a generic parameter $\FI \in \mathfrak{g}^\vee_\mathbb{R}$. As described in \cite{kashiwara2008microlocalization}*{\S 2}, the reduction carries a sheaf $\cW$ of $\mathbb{C}((\hbar))$-algebras, which can be viewed as a microlocalization of the sheaf of differential operators. 

To define category $\cO$, one considers two $\mathbb{C}^\times$ actions on $ \htvar_G(\FI)$. First, consider the action on $T^*E=E\times E^\vee$ given by scalar multiplication.
Second, fix a Hamiltonian $\mathbb{C}^\times$-action $\mass$ which factors through $\operatorname{Aut}_G(E)/Z(G)$. To avoid confusing these actions, we follow the convention of \cite{BPW} and denote the first copy of $\mathbb{C}^\times$ acting holomorphically by $\mathbb{S}$;
we will often refer to the second as $\CC^\times_m.$

We now form the \emph{category $\cO$ skeleton} 
\begin{equation}\label{eq:oskel-1}
     \Oskel:= \{ p \in \htvar_G(\FI) \mid \lim_{s\to 0} s \cdot_\mass p \text{ exists}\} \subset  \htvar_G(\FI).
 \end{equation}
  Thanks to the calculations in \Cref{sec:Lefschetz}, this Lagrangian may be understood as the relative skeleton associated to the Lefschetz fibration given by the $J$-holomorphic moment map for the $\CC^\times_m$ action.



\begin{definition}[\cite{BLPW16}] \label{definition:de-rham-O}
    The {\em (de Rham) category $\cO$} is the category of coherent, $\mathbb{S}$-equivariant $\cW$-modules which are supported in $\LL_G(\FI,\mass)$ and admit a good filtration.%
    \footnote{In \cite{BLPW16}, this category is called the ``geometric category $\cO$'' and denoted by $\mathcal{O}_g$. There is also an ``algebraic category $\cO$'' considered there which will not be discussed in this paper. 
    For the notion of a good filtration, see \cite[\S 2.5]{BPW}.}
    We denote it by $\cO^{\dR}_G(\FI,\mass).$
\end{definition}
 Work of many authors including \cite{kashiwara2008microlocalization,BLPW16} proves that this abelian category enjoys many of the same wonderful properties as the classical BGG category $\cO$. It also recovers the BGG category $\cO$ in the case where $\fX_G(\FI)$ is the cotangent bundle of a flag variety (in which case the category $\cO$ skeleton is the union of conormals to Schubert strata).

Absent so far from this theory, outside the setting of $T^*(G/B),$ has been a Betti counterpart to the de Rham category $\cO$. (See \cite{Jin-bigtilting} for an early Fukaya-categorical perspective on the classical category $\cO.$) One reason for this absence is that there did not exist a general theory of ``microlocal perverse sheaves'' on which to base this definition. This missing ingredient has now been supplied by \cite{perverse-microsheaves}, which allows us to consider the following:
\begin{definition}\label{definition:betti-O}
    The {\em Betti category $\cO$} is the category $\cO^{\Bet}_G(\FI, \mass):=(\mush_{\Oskel}(\Oskel)^c)^{\heartsuit}$ of microlocal perverse sheaves on $\htvar_G(\FI)$ microsupported in $\Oskel$.\footnote{The superscript $(-)^c$ denotes passage to compact objects. We refer to \Cref{sec:geometric-categories} for a discussion of (microlocal perverse) sheaves. Throughout the introduction, we also suppress in our notation the polarization data used to define $\mush_{\Oskel},$ which we describe in \Cref{def:Oskel-polarization}.}
\end{definition}




\subsection{Hyperkähler toric manifolds}
We shall prove that the Betti and de Rham category $\cO$ are equivalent for an important class of symplectic resolutions called \emph{hyperkähler toric manifolds} (or hypertoric manifolds). 

Hyperkähler toric manifolds are defined by performing hyperkähler reduction on a torus action on $T^*\mathbb{C}^n$. More precisely, consider an exact sequence of complex tori
\begin{equation}\label{equation:exact-sequence-intro}
    \begin{tikzcd}
    1 \ar[r]& G \ar[r]& D\ar[r]&  F \ar[r]& 1,
    \end{tikzcd}
\end{equation}
and fix an identification $D \simeq (\mathbb{C}^\times)^n$. The standard linear action of $D$ on $\mathbb{C}^n$ induces a linear action of $G$. Now consider the induced Hamiltonian action on $T^*\mathbb{C}^n$ and form the hyperkähler reduction $\htvar_G(\FI)$ with respect to a stability parameter $\FI= (\FI, 0,0) \in \mathfrak{d}_\RR^\vee\otimes \RR^3$, where $\mathfrak{d}$ is the Lie algebra of $D$. Observe that  $\htvar_G(\FI)$ admits a residual Hamiltonian $F$ action. We finally fix a subtorus $\mathbb{C}^\times \subset F$, determined by a cocharacter $m$ of $F,$ which allows us to define the category $\cO$ skeleton $\Oskel\subset \fX_G(\FI)$.

The de Rham category $\cO$ in this setting was studied in \cites{BLPW10,BLPW12}, and it was shown \cite{BLPW12}*{Theorems 4.7, 4.8} that $\cO^{\dR}_G(\FI,\mass)$ was equivalent to the category of modules of a certain (discrete, i.e., non-derived) $\CC$-algebra, which we will denote henceforth by $A_G^{\dR}(\FI,\mass)$ and which may be understood as the endomorphism algebra of the projective objects. The main result of this paper is an identification of their calculation with our Betti category $\cO.$

\begin{theorem}\label{theorem:main-theorem}
  With respect to the above data for a hyperkähler toric manifold, there is an equivalence of stable categories \begin{equation}\label{equation:main-theorem-equivalence}
      \mush_{\Oskel}(\Oskel)  \simeq \Mod_{A^{\dR}_G(\FI, \mass)}.
 \end{equation}
This equivalence intertwines the (microlocal) perverse t-structure on the left hand side with the natural t-structure on the right hand side.
\end{theorem}

Passing to compact objects and taking hearts, one obtains:
\begin{corollary}\label{corollary-main-abelian}
  $\mathcal{O}^\Bet_G(\FI, \mass)  \simeq \mathcal{O}^{\dR}_G(\FI, \mass)$.
\end{corollary}

\begin{remark}
    In the above results, in order to compare with the de Rham category, we have implicitly taken $\CC$ as the coefficient ring for our Betti category $\cO.$ However, unlike in the de Rham setting, this choice is not forced upon us: our definition of Betti category $\cO$ as a category of microlocal sheaves allows us to work with integral coefficients. This is a major source of added richness in the Betti theory, and it hints at the presence of a categorification, constructed in \cite{GH-2O}. The equivalence with the de Rham theory is recovered after base change to $\CC,$ by a map which is a priori complex-analytic but actually, due to strong finiteness properties of category $\cO,$ turns out to be algebraic.
\end{remark}

\begin{remark}
    The presence of a t-structure on the left-hand category in \Cref{equation:main-theorem-equivalence} follows from \cite{perverse-microsheaves}, but \emph{it does not follow from general principles} that this category is the derived category of its heart. As a result, one cannot expect to deduce \Cref{theorem:main-theorem} directly from \cite{BLPW12} and a Riemann--Hilbert correspondence such as \cite{cote2024microlocal}.
\end{remark}

\subsection{Symplectic duality/3d mirror symmetry}\label{ssec:3dhms}
If we replace the exact sequence \Cref{equation:exact-sequence-intro} with its dual
\[
\begin{tikzcd}
    1&\ar[l] G^\vee&\ar[l]D^\vee&\ar[l] F^\vee&\ar[l]1,
\end{tikzcd}
\]
then the parameters $(\FI,\mass)\in \fgv_\RR\times \frf_\RR$ switch their roles; this is an incarnation of the combinatorial phenomenon of {\em Gale duality}.
Let $S^\dR$ be the direct sum of simple objects in $\cO^{\dR}_G(\FI,\mass)$ and $P^\dR$ the direct sum of their projective covers, so that there is an isomorphism of algebras $A_G^{\dR}(\FI,\mass)\cong \End(P^\dR)$; dually, let $(S^\vee)^{\dR}$ and $(P^\vee)^\dR$ be the same for $\cO^{\dR}_{F^\vee}(\mass,\FI).$
Then the main results of \cite{BLPW10} are:
\begin{enumerate}
    \item The algebras $\End(P^\dR)$ and $\End(S^\dR)$ are Koszul dual.
    \item The algebras $\End(P^\dR)$ and $\End((S^\vee)^\dR)$ are equivalent (up to a grading shear).
\end{enumerate}

We now have a clearer understanding of such dualities, thanks to recent developments \cite{BDGH} situating the above in the theory of supersymmetric 3d gauge theories. We can now refine the second equivalence from \cite{BLPW10};
we claim the most natural statement relates a de Rham category $\cO$ to the Gale dual Betti category $\cO.$
The following is a combination of our main theorem with \cite{BLPW10}*{Theorems A,B}:
%
\begin{corollary}\label{corollary:3d-mirror-intro}
    Let $P^\Betti$ be the direct sum of projective objects in 
    $\cO^\Betti_G(\FI,\mass)$
    and $(S^\vee)^\dR$ the direct sum of simple objects in 
    $\cO^\dR_{F^\vee}(\mass,\FI).$ 
    Then the algebras $\End_{\cO^\Betti_G(\FI,\mass)}(P^\Betti)$ and $\End_{\cO^\dR_{F^\vee}(\mass,\FI)}((S^\vee)^{\dR})$ are equivalent, up to a grading shear.
\end{corollary}
Our insistence on the statement of \Cref{corollary:3d-mirror-intro} 
comes from the theory of 2-categorical invariants of hyperk\"ahler manifolds. The homological 3d mirror symmetry program \cite{GMH,GH-2O} predicts that if $\fX,\fX^\vee$ are a dual pair \cite{BLPW16} of symplectic resolutions, then there is an equivalence
\begin{equation}\label{eq:3dhms}
\Athree\simeq\Bthree
\end{equation}
between a pair of 2-categories with the following properties:
\begin{itemize}
    \item We can reduce \Cref{eq:3dhms} to an equivalence of 1-categories by applying the periodic cyclic homology functor $\HP(-)$. The resulting category $\HP(\Athree)$ will be a category of microlocal (perverse) sheaves on $\cX$, while $\HP(\Bthree)$ will be a category of holonomic DQ-modules on $\fX^\vee$.
    \item Under the equivalence \Cref{eq:3dhms}, the objects of $\cA_\fX$ which become projective objects in the microsheaf category $\HP(\Athree)$ are sent to objects which become simple objects in the DQ-module category $\HP(\Bthree).$
\end{itemize}
We conclude that it is \Cref{corollary:3d-mirror-intro} (or at least a $\ZZ/2$-graded version) which naturally follows from the extended 3d mirror equivalence \Cref{eq:3dhms} and the above facts. We refer to \cite{GH-2O} for more details about this proposal, and a proof of the above facts in the hypertoric setting.

\subsection{Fukaya categories}\label{subsection:intro-fukaya}
The celebrated work of Ganatra--Pardon--Shende \cite{GPS3} furnishes an (anti)equivalence of categories
\begin{equation}\label{equation:gps-equivalence-intro}
    \mush_{\Oskel}(\Oskel)^c \simeq \cW(\htvar_G(\FI), \partial_\infty \LL_G(\FI,\mass))^{op},
\end{equation}
where, for $X$ a Weinstein manifold and $\Lambda\subset \partial_\infty X$ a Legendrian, we write $\cW(X,\Lambda)$ for the idempotent-completed, pretriangulated closure of the Fukaya category of $X$ stopped at $\Lambda,$ as defined in \cite{GPS1}.
Combining \eqref{equation:gps-equivalence-intro} with \Cref{theorem:main-theorem}, we may immediately read off structural properties of Fukaya categories from 
known statements about hypertoric category $\cO$ established in \cite{BLPW12}.
\begin{corollary}\label{corollary:main-fukaya-intro}
        There is an equivalence of stable categories $\cW(\htvar_G(\FI), \partial_\infty \Oskel)^{op} \simeq \Perf_{A^{\dR}_G(\FI, \mass)}$ which intertwines the perverse t-structure with the standard t-structure. As a result, we deduce the following facts from \cite{BLPW12}:
    \begin{enumerate}
        \item $\cW(\htvar_G(\FI), \partial_\infty \Oskel)$ admits a canonical full exceptional collection (corresponding heuristically to the thimbles of a Lefschetz fibration).
       \item The complexified Grothendieck group of $\cW(\htvar_G(\FI), \partial_\infty \Oskel)$ is isomorphic to $H^{2n}_{\CC^\times_m}(\htvar_G(\FI)).$
        \item There are two natural collections of commuting auto-equivalences on $\cW(\htvar_G(\FI), \partial_\infty \Oskel)$ called \emph{twisting} and \emph{shuffling}, which involve moving in the respective (complexified) moduli spaces for regular parameters $\FI$ and $\mass.$
    \end{enumerate}
\end{corollary}
As $A^{\dR}_G(\FI, \mass)$ is known to be finite dimensional over $\mathbb{C}$ and of finite global dimension, \Cref{corollary:main-fukaya-intro} is a very strong finiteness condition on the Fukaya category. 


On the other hand, we can also leverage our Fukaya-categorical perspective for new insights about category $\cO$; for instance, by applying the main theorem of \cite{ST-CY}, we may deduce the following fact, which does not seem to have been noted before in the literature (possibly because it involves the Fukaya category of a real symplectic manifold). Let $\frF$ be a fiber of the Lefschetz fibration $\htvar_G(\FI)\to\CC$ descibed in \Cref{sec:Lefschetz}.
\begin{corollary}
    The category $\cO_G^{\Betti}(\FI,\mass)$ (hence also $\cO_G^{\dR}(\FI,\mass)$) has a restriction morphism to the wrapped Fukaya category 
    $\cW(\frF),$
    and this functor has a relative Calabi-Yau structure.
    By \cite{Brav-Dyckerhoff-2}, we conclude that the moduli space of objects in $\cO_G^\Bet(\FI,\mass)$ admits a Lagrangian morphism to the moduli space of objects in $\cW(\frF).$
\end{corollary}

We also note a Fukaya-categorical consequence of the 3d mirror duality discussed above. 
Let $\ocW(-)$ denote the ``2-periodicized'' Fukaya category, obtained from $\cW(-)$ by base-change along $\ZZ\to\ZZ[u^\pm],$ where $u$ has cohomological degree 2.
\begin{corollary}\label{corollary:3d-intro} 
There is an equivalence
\begin{equation}\label{eq:3dhms-fukaya-intro}
\ocW(\htvar_G(\FI),\partial_\infty \LL_G(\FI,\mass)) \simeq \ocW(\htvar_{F^\vee}(\mass), \partial_\infty \LL_{F^\vee}(\mass, \FI))
\end{equation}
between the (2-periodicized) partially-wrapped Fukaya categories associated to Gale dual hypertoric varieties.
\end{corollary}

\begin{example}
    If $G=\CC^\times\xrightarrow{\Delta}(\CC^\times)^{n+1} = D$ is the diagonal torus, then 
    there are natural identifications
    $
    \fX_G(\FI) = T^*\PP^n,$ $\htvar_{F^\vee}(\mass) = \widetilde{\CC^2/(\ZZ/n)}.
    $
    These varieties are not even of the same dimension, and their compact skeleta are very different: the former is just $\PP^n$, whereas the latter is a nodal chain of $n$ $\PP^1$'s. On the other hand, after equipping each manifold with a natural Lefschetz fibration, the former skeleton acquires $n$ more noncompact components, namely conormals to the Schubert varieties $\{0\}\subset \PP^1\subset \cdots \subset \PP^{n-1}\subset \PP^n,$ and the latter acquires a single noncompact component. 3d mirror symmetry explains the miraculous fact that the Ext algebra of one collection of these simple Lagrangians is equivalent, after a grading shear, to the endomorphism algebra of cocores in the Gale dual category.
\end{example}

\begin{remark}
    The equivalence \Cref{eq:3dhms-fukaya-intro} comes from the equivalence between endomorphism algebras of respective generators $P^\Bet$ and $(S^\vee)^\Betti$ (in the notation of \Cref{ssec:3dhms}). The need for 2-periodicization arises because these algebras are isomorphic only after a shear of grading.
    (The existence of such a grading shear, putting the whole Ext algebra of simples into degree 0, is rather remarkable and heavily constrains the structure of the algebra, as we see in the following corollary.)
\end{remark}

In \Cref{subsec:formality}, we also appeal to the Koszulity established in \cite{BLPW12} to deduce the formality of the Floer--Fukaya algebra for components of the skeleton of a hypertoric variety. 

%
%

\begin{corollary}
\label{corollary:formality-intro}
Let $S$ be the direct sum of the irreducible components of $\Oskel$. 
Then $CF^*(S,S)$ is formal in characteristic zero. 
\end{corollary}

Finally, in \Cref{ssec:Koszul}, we prove a new Koszul duality statement about the wrapped Fukaya category of a hypertoric variety $\htvar_G(\FI).$ As previously mentioned, the Koszul duality statements of \cites{BLPW10,BLPW12} are between the simple and projective objects of the category $\cO^{\dR}_G(\FI,\mass)$; using our main theorem, we may reinterpret this as a Koszul duality in the partially wrapped Fukaya category $\cW(\htvar_G(\FI),\partial_\infty\Oskel)$ between the components of $\LL_G(\FI,\mass)$ and their cocores. However, by an analysis in \Cref{ssec:koszul-abstract} of how Koszul duality interacts with idempotents, we may deduce from this a Koszul duality within the fully wrapped Fukaya category $\cW(\htvar_G(\FI))$:
\begin{theorem}\label{thm:koszul-intro}
    Let $\Core$ be the direct sum in $\cW(\htvar_G(\FI))$ of components of the skeleton of $\htvar_G(\FI),$ and $\Cocore$ the direct sum of the corresponding cocores. Then the algebras $\End_{\cW_G(\FI)}(\Core)$ and $\End_{\cW_G(\FI)}(\Cocore)$ are Koszul bidual in the sense of \Cref{def:koszul-dual}.
\end{theorem}
When $\htvar_G(\FI)$ is the resolution of an $A_n$-singularity, \Cref{thm:koszul-intro} recovers a result of \cite{etgu2017koszul} in type $A_n.$ (Note that \cite{etgu2017koszul} also establishes Koszul duality in type $D_n$, and the calculates the Fukaya category for general tree plumbings.)   One virtue of this theorem is that while the endomorphism algebra of cocores, which involves wrapping, may be difficult to compute, the endomorphism algebra of $\Core$ is a finite-dimensional and formal dg-algebra: if we write $\Core = \bigoplus_\alpha S_\alpha$ for the components of the skeleton, then this algebra may be written as a matrix algebra with its entries given by cohomology of intersections,
\begin{equation}\label{eq:simples-algebra-intro}
\End_{\cW_G(\FI)}(\Core) \simeq \bigoplus_{\alpha,\alpha'} H^*(S_\alpha\cap S_\alpha'),
\end{equation}
where multiplication is given by a convolution using triple intersections. Analogously to \cite{etgu2017koszul}, we learn from \Cref{thm:koszul-intro} that the symplectic cohomology of the manifold $\htvar_G(\FI)$ may be computed in terms of the algebra \Cref{eq:simples-algebra-intro}:
\begin{corollary}
    There is an equivalence
    \[
    SH^{*+\dim(\htvar_G(\FI))}(\htvar_G(\FI))\simeq HH_{-*}(\End_{\cW_G(\FI)}(\Core))^\vee
    \]
    between the symplectic cohomology of $\htvar_G(\FI)$ and the linear dual of the Hochschild homology of the algebra \Cref{eq:simples-algebra-intro}.
\end{corollary}

We note that the phenomenon of Koszul duality in symplectic geometry has been studied earlier in e.g.\ \cites{etgu2017koszul,li2019exact,li2019koszul}; guided by those investigations, we outline in \Cref{sec:fukaya} several further consequences of \Cref{thm:koszul-intro}. One such consequence which may appeal to geometrically-minded readers is the following:
\begin{corollary}\label{corollary:exact-lags}
    $\htvar_G(\FI)$ does not contain an exact Lagrangian $K(\pi, 1)$.
\end{corollary}
This is a standard consequence of the existence of a dilation on $SH^*(\htvar_G(\FI))$, which we will deduce from \Cref{thm:koszul-intro}. In dimension $4$, \Cref{corollary:exact-lags} recovers a result of Ritter \cite{ritter2010deformations}, who proved more generally that the only exact Lagrangians in ALE spaces are spheres. 

We also refer to \Cref{sec:fukaya} for more discussion of the Lefschetz fibration perspective on category $\cO.$

\begin{remark}
    Forthcoming work of Lee, Li, Mak \cite{LLM} will provide an alternative approach to studying Fukaya categories of hypertoric varieties which does not go through \cite{GPS3}. 
    Applications of their work include a different proof of \Cref{corollary:formality-intro}, a computation of the endomorphism algebra of the simples in $\cW(\htvar_G(\FI), \partial_\infty \LL_G(\FI,\mass))$, and a proof that they generate the infinitesimally wrapped Fukaya category.
\end{remark}

\begin{remark}
    Another approach to the Fukaya--Seidel category of a hyperk\"ahler manifold equipped with a J-holomorphic moment map is described in \cite{Khan-momentmap}. 
    However, that approach, based on the ``algebra of the infrared'' formulated in \cite{GMW},
    requires that the critical values of the moment map are in general position,
    and thus it cannot be directly applied to category $\cO$, for which all critical values lie on a line. 
    This restriction can be eliminated using the formulation of the Fukaya-Seidel category in \cite{Haydys}.
\end{remark}

\subsection{Outline of the proof} 
%
%
%

Our calculation of the Betti category $\cO$ proceeds in several steps. First, we recall the classical calculation (phrased in the language of sheaves rather than microsheaves) of microsheaves on the Lagrangian $\LL\subset T^*\CC^n$ given by the union of conormals to toric strata. The $G$-action on this category is also well understood, so that we may present the equivariant microsheaf category
\[
\mush^G_\LL(\LL)\simeq \Mod_{A_G^\Betti}
\]
as the category of modules over an algebra $A_G^\Betti$ with $2^n$ commuting idempotents, corresponding to the components of the Lagrangian $\LL_G:=\LL/G.$

Next, we pass from the stacky Lagrangian $\LL_G$ to an open non-stacky subset $\LL_G(\FI,-)$ of points which are stable with respect to the reduction parameter $\FI\in \fgv_\RR.$ Writing $e_\cF$ for the sum of idempotents in $A_G^\Betti$ corresponding to components of $\LL_G(\FI,-)\subset \LL_G,$ we obtain a new algebra $A_G^\Betti(\FI,-):=e_\cF A_G^\Betti e_\cF,$ and
we exhibit an equivalence 
\begin{equation}\label{equation:intro-global-sec}
    \mush_{\LL_G}(\Lstab)\simeq \Mod_{A_G^\Bet(\FI,-)}.
\end{equation}
We establish \eqref{equation:intro-global-sec} by exhibiting both sides as the global sections of equivalent sheaves of categories, partly imitating a similar argument from \cite[Sec.\ 4.3]{gammage2019homological}. 
The key technical step is to prove that the presheaf of categories $\Mod_{e_\alpha A_G^\Betti(\FI,-) e_\alpha}$ satisfies descent. This statement is established in a joint work of the authors with Michael McBreen and Ben Webster, located in \cite[Appendix C]{gammage2019homological}.  
%


Finally, we pass from $\LL_G(\FI,-)$ to the closed subset $\LL_G(\FI,\mass)$ defined by \Cref{eq:oskel-1}: by stop removal, the category $\mush_{\LL_G}(\LL_G(\FI,\mass))$ may be obtained from $\mush_{\LL_G}(\LL_G(\FI,-))\simeq \Mod_{A_G^\Betti(\FI,-)}$ as the categorical quotient by the objects corepresenting microstalks at the components of the complement $\LL_G(\FI,\mass)\setminus \LL_G(\FI,-).$ Writing $e_{\cF\cap cU}\in A_G^\Betti(\FI,-)$ for the idempotent corresponding to these components, we define 
\begin{equation}\label{eq:underived-quotient-intro}
    A_G^\Betti(\FI,\mass):=
\frac{A_G^\Betti(\FI,-)}{A_G^\Betti(\FI,-)e_{\cF\cap \cU}A_G^\Betti(\FI,-)}
\end{equation}
to be the quotient algebra.

This definition does not yet guarantee us an equivalence $\mush_{\LL_G}(\Oskel)\simeq \Mod_{A_G^\Betti(\FI,\mass)},$ since \Cref{eq:underived-quotient-intro} is defined as a naive (underived) quotient of algebras, whereas the categorical quotient category $\mush_{\LL_G}(\Oskel)$ is controlled by a derived quotient described in \Cref{defn:derived-quotient}. We must therefore establish that the derived and underived quotients agree in this case. 

To understand the quotient \Cref{eq:underived-quotient-intro},
we write down a Riemann--Hilbert-type map, defined on a certain completion of $A_G^\Betti,$ to identify it with the completion of another algebra $A_G^{\dR}$ related to the category of $G$-equivariant DQ-modules on $\LL.$
The completion is necessary because the Riemann--Hilbert correspondence is controlled by a map which is not algebraic but only complex-analytic. Luckily, this completion still contains enough information to recover the construction of $A_G^\Betti(\FI,\mass)$ and
(by following the Riemann--Hilbert isomorphism) to produce an isomorphism
\begin{equation}\label{eq:intro-proof-agbetti-is-derham}
A_G^\Betti(\FI,\mass)\simeq A_G^{\dR}(\FI,\mass)
\end{equation}
with the de Rham category $\cO$ algebra studied in \cites{BLPW10,BLPW12}.
We can then leverage many of the properties of de Rham category established in \cite{BLPW12}, deducing in particular that the quotient \Cref{eq:underived-quotient-intro} is in fact a derived quotient. Miraculously, although the original Riemann--Hilbert map we wrote down was not algebraic, its restriction to \Cref{eq:intro-proof-agbetti-is-derham} does in fact define an algebaic isomorphism of $\CC$-algebras, due to the strong finiteness properties of category $\cO.$

\subsection{Acknowledgements}
We benefited from useful discussions and suggestions from many mathematicians, including  Roman Bezuravnikov, Benjamin Duenzinger, Georgios Dimitroglou Rizell, Christopher Kuo, Maxence Mayrand, David Nadler, Maxime Ramzi, Lisa Sauermann, Vivek Shende, Germán Stefanich, Minh-Tâm Trinh and Filip Živanović. We are particularly grateful to Michael McBreen and Ben Webster for collaboration on \cite{gammage2019homological}*{Appendix C}, which was developed during the preparation of this paper.

Part of this work was conducted while LC and BG were supported by NSF grant DMS-2305257, and while LC was a member of the Max Planck Institute for Mathematics, which he warmly thanks for its hospitality. 

This research was supported in part by Perimeter Institute for Theoretical Physics. Research at Perimeter Institute is supported by the Government of Canada through the Department of Innovation, Science and Economic Development Canada and by the Province of Ontario through the Ministry of Research, Innovation and Science.


\section{Hyperk\"ahler toric geometry}\label{sec:hyperplane}

\subsection{Hyperplane arrangements} \label{subsection:hyperplane-arrangements} We review hyperplane arrangements for the purpose of setting our notation. We mostly follow \cite[Sec.\ 2]{BLPW10}, to which we refer the reader for details and motivation.  

\begin{definition}
    Let $n\geq 1.$ A {\em sign vector} is an element of $2^{[n]}\cong \{+,-\}^n\cong \cP([n]),$ where the latter isomorphism takes $\alpha=(\alpha_1,\ldots,\alpha_n)\in \{+,-\}^n$ to the subset $\{k\in [n]\mid \alpha_k = +\}.$
\end{definition}


\begin{definition}
A \emph{hyperplane arrangement} is an $n$-dimensional, oriented, real vector space $V$ and a collection $\{ H_1, \dots, H_k\}$ of $k \geq 0$ oriented affine-linear hyperplanes (empty if $k=0$). 
\begin{itemize}
    \item  The connected components of $V- \bigcup_i H_i$ are called \emph{chambers}. 
    \item Given $S \subset [k]= \{1,\dots,k\}$, the intersection $H_S:= \bigcap_{i \in S} H_i$ is called the \emph{flat spanned by $S$}.
    \item A hyperplane arrangement is \emph{simple} if $H_S$ has codimension $|S|$ for all $S \subset [k]$. 
\end{itemize}
\end{definition}

\begin{example}
    The coordinate hyperplanes
$H_i := \{(x_1,\ldots,x_n)\in \RR^n\mid x_i =0\}$ form a hyperplane arrangement in $\RR^n$. This arrangement is simple, and the chambers are naturally indexed by $2^{[n]}$. We will write $\Delta_\alpha$ for the chamber corresponding to $\alpha\in 2^{[n]}.$
\end{example}

\begin{definition}
    A {\em polarized} hyperplane arrangement (indexed by $[n]$) is a triple $(V,\FI,\mass),$ where $V\subset \RR^n$ is a linear subspace, $\FI\in \RR^n/V, \mass\in V^\vee.$ Throughout this paper, we will make the assumption that the inclusion $V\hookrightarrow \RR^n$ is induced (by $\otimes_\ZZ\RR$) from an inclusion of lattices $V_\ZZ\hookrightarrow \ZZ^n.$
\end{definition}
The underlying hyperplane arrangement corresponding to the above data $(V,\FI,\mass)$ lives in the affine space $V(\FI):=V+\FI\subset \RR^n,$ and it is given by the $n$ affine-linear hyperplanes $H_i\subset V(\FI)$ defined by $H_i=\{v=(v_1,\ldots,v_n)\in V(\FI)\subset \RR^n\mid v_i=0\}.$

\begin{definition}
Let $(V,\FI,\mass)$ be a polarized hyperplane arrangement.
    \begin{itemize}
        \item We say that $\FI$ is regular if the corresponding hyperplane arrangement is simple, and we say that $\mass$ is regular if it is not constant on any 1-dimensional flat. We call the polarized arrangement $(V,\FI,\mass)$ {\em regular} if both $\FI$ and $\mass$ are regular.
        \item  A sign vector is said to be ($\FI$-)\emph{feasible} if the corresponding chamber in $\mathbb{R}^n$ has non-empty intersection with $V_\FI$; otherwise it is called infeasible. We write $\cF\subset 2^{[n]}$ for the collection of feasible sign vectors.
        \item A sign vector is said to be ($\mass$-)\emph{bounded} if the restriction of $\mass$ to the corresponding chamber in $V$ is bounded; otherwise we say it is unbounded. We write $\cB\subset 2^{[n]}$ for the collection of bounded sign vectors. 
        \item The hyperplane arrangement is {\em unimodular} if for all $I\subset [n],$ the image of $V_\ZZ\hookrightarrow \ZZ^n \twoheadrightarrow \ZZ^I$ is a direct summand of $\ZZ^I.$
    \end{itemize}
    \end{definition}



\subsection{Quaternionic vector spaces and reduction} \label{subsection:data-hypertoric} \label{subsection:hypertorics}
The spaces we study will be obtained from hyperk\"ahler reductions of quaternionic vector spaces.
We write 
$
\HH = \{x_0+Ix_1+Jx_2+Kx_3\mid x_i\in \RR\}
$
for the quaternions. The coordinates $x_i$ give an isomorphism
$
\HH\cong \RR^4.
$
By distinguishing the complex structure $I,$ we can identify $\HH\cong \CC^2$ by declaring that $(z,w):=(x_0+ix_1, x_2+ix_3)$ are holomorphic coordinates. These constructions naturally extend to $\HH^n\cong \HH\otimes_\RR \RR^n.$


We write $g(-,-)$ for the metric on $\HH^n$ induced by the Euclidean metric on $\RR^{4n}.$ This is a hyperk\"ahler metric; for $\cJ\in \{I,J,K\}$ we write
$
\omega_\cJ:=g(\cJ-,-)
$
for the corresponding symplectic form, which is the standard K\"ahler form on the $\cJ$-complex K\"ahler manifold $\CC^{2n}.$
We write $\Omega_I:=\omega_J+i\omega_K$ for the $I$-holomorphic symplectic form.
When we consider real symplectic geometry in this paper, it will be with respect to the real symplectic form $\omega_J:=\Re(\Omega_I).$
\begin{remark}
    Although the complex structure $I$ has a privileged status in our constructions, the form $\omega_J$ is not actually distinguished, and could be replaced throughout this paper by any $\omega_{J_\theta}:=\Re(e^{i\theta}\Omega_I)$. Indeed, the richest content of holomorphic Floer theory \cite{KS-holomorphic-floer} is only visible when considering this whole $S^1$ family of symplectic forms. 
\end{remark}

It will be convenient to identify $T^*\mathbb{C}^n \simeq \mathbb{C}^{2n}$ via the map 
\begin{equation}\label{equation:id-tc-h}
(q_1,\dots, q_n, p_1, \dots, p_n) \mapsto (q_1, -p_1, \dots, q_n, -p_n).
\end{equation}
Note that per our conventions, the symplectic form on $T^*Q$ is $d\lambda_{can}$ (if we had used instead $-d\lambda_{can}$, then there would be no minus signs in \eqref{equation:id-tc-h}).


Let $G$ be a complex reductive group acting linearly on $\mathbb{C}^n$, with maximal compact subgroup $G_c$. The following is well-known: 
\begin{lemma}
    The action of $G_c$ on $T^*\CC^n$ is  trihamiltonian, in the sense that the action is Hamiltonian with respect to each of the symplectic forms ${\omega_I,\omega_K,\omega_J}.$
\end{lemma}
Using the $I$-complex structure, we may split the resulting triple of moment maps
\begin{equation}\label{eq:basic-tri-moment}
\mu:\HH^n\to \mathfrak{d}_\RR^\vee\otimes \RR^3
\end{equation}
into a real and complex part,
\[
\mu^\RR:\HH^n\to \mathfrak{d}_\RR^\vee
\qquad\text{and} \qquad 
\mu^\CC:\HH^n\to \mathfrak{d}^\vee.
\]
\begin{example}\label{ex:toric-moment-maps}
    Let $G=(\CC^\times)^n,$ with its standard action on $\CC^n.$ 
with its standard action on $\CC^n.$ 
The real and complex moment maps for $G$ may be written in coordinates as
\[
(x_1,\ldots,x_n,y_1,\ldots,y_n)\mapsto (|x_1|^2-|y_1|^2,\ldots,|x_n|^2-|y_n|^2),\qquad
(x_1,\ldots,x_n,y_1,\ldots,y_n)\mapsto (x_1y_1,\ldots,x_ny_n),
\]
respectively.
\end{example}
\begin{theorem}[\cite{HKLR}*{Theorem 3.1}]\label{thm:hklr-reduction}
    Suppose that $\xi \in \mathfrak{d}^\vee \otimes \mathbb{R}^3$ is fixed by the coajoint action of $G$. If $G$ acts acts freely on $\mu^{-1}(\xi)$, then there is a canonical hyperkähler structure $(\overline{\omega}_I, \overline{\omega}_J, \overline{\omega}_K)$ on $\mu^{-1}(\xi)/G_c$. It is uniquely determined by the property that so that $\pi^* \overline{\omega}_I= \iota^* \omega_I$ (resp.\ for $J, K)$), where $\iota: \mu^{-1}(0) \hookrightarrow \mathbb{H}^n$ is the inclusion and $\pi: \mu^{-1}(0) \to \mu^{-1}(0)/G_c$ is the projection.
\end{theorem}

By splitting the moment maps up using the $I$-complex structure as above, we recover an $I$-complex algebraic perspective on the hyperk\"ahler reduction described in \Cref{thm:hklr-reduction}.
\begin{proposition}
    As an $I$-complex algebraic variety, the hyperk\"ahler reduction $\mu^{-1}(\xi)/G_c$ is isomorphic to the GIT quotient $(\mu^{\CC})^{-1}(\xi_2+i\xi_3)\redu_{\xi_1}G_c.$
\end{proposition}
As a result, we have a very useful embedding of the hyperk\"ahler reduction $\mu^{-1}(\xi_1,0,0)$ into an Artin stack:
\begin{corollary}\label{cor:substack-embedding}
    The $I$-holomorphic symplectic manifold $\mu^{-1}(\xi_1,0,0)/G_c$ admits an open symplectic embedding
    \[
    \mu^{-1}(\xi_1,0,0)/G_c\hookrightarrow T^*(\CC^n/G) = (\mu^\CC)^{-1}(0)/G
    \]
    as the $\xi_1$-semistable locus inside the stack $(\mu^\CC)^{-1}(0)/G.$
\end{corollary}



\subsection{Toric hyperk\"ahler manifolds}
We now specialize to the case of interest in this paper. 
Fix an exact sequence of tori
\begin{equation}\label{eq:basic-exact-sequence}
\begin{tikzcd}
1 \arrow[r] & G \arrow[r,"i"]& D \arrow[r,"p"] & F \arrow[r] & 1,
\end{tikzcd}
\end{equation}
together with an identification $D\cong (\CC^\times)^n.$
The exact sequence \Cref{eq:basic-exact-sequence} induces exact sequences
\[
\begin{tikzcd}
0 \arrow[r]& \mathfrak{g}_{\ZZ} \arrow[r,"i_{\ZZ}"]& \mathfrak{d}_{\ZZ} \arrow[r,"p_{\ZZ}"]& \mathfrak{f}_{\ZZ} \arrow[r]& 0,
\end{tikzcd}\qquad
\begin{tikzcd}
0 &\arrow[l] \mathfrak{g}^{\vee}_{\ZZ}& \arrow[l,"i^\vee_{\ZZ}"'] \mathfrak{d}^{\vee}_{\ZZ} &\arrow[l,"p^\vee_{\ZZ}"'] \mathfrak{f}^{\vee}_{\ZZ} &\arrow[l] 0
\end{tikzcd}
\]
of cocharacter and character lattices, respectively. The identification $D \cong (\CC^\times)^n$ determines identifications $\mathfrak{d}_{\ZZ} \cong \oplus_{i=1}^n \ZZ e_i$ and $\mathfrak{d}^{\vee}_{\ZZ} \cong \oplus_{i=1}^n \ZZ e^i$. In the rest of this subsection we will abbreviate tensor products $(-) \otimes_{\ZZ} \RR$ by changing the subscript $\ZZ$ to $\RR$.

\begin{definition}
A \emph{stability parameter} is an element $\FI \in \mathfrak{g}^{\vee}_{\RR}$. A {\em mass parameter} is an element $\mass\in \mathfrak{f}_\RR.$
\end{definition}

The exact sequence \Cref{eq:basic-exact-sequence} determines the embedding $\ffv_\RR\xrightarrow{p_\RR^\vee}\RR^n$; together with the parameters $\FI,\mass,$ it therefore determines a polarized hyperplane arrangement $(\ffv_\RR,\FI,\mass).$

\begin{definition}\label{defn:H(t)}
    We write $\ffv_\RR(\FI):=\ffv_\RR+\FI$ for the affine-linear subspace of $\RR^n$ determined by the parameter $\FI,$ and $\cH(\FI)$ for the hyperplane arrangement in $\ffv_\RR(\FI)$ induced by the cooridnate hyperplane arrangement on $\mathfrak{d}^\vee_\RR\cong\RR^n.$
\end{definition}


%

\begin{definition}
    A {\em (hypertoric) category $\cO$ datum} is a polarized hyperplane arrangement indexed by $[n],$ which we take to be determined by the exact sequence \Cref{eq:basic-exact-sequence} together with stability and mass parameters $(\FI,\mass)\in\fgv_\RR\times \frf_\RR,$ which is both regular and unimodular.
\end{definition}
For the remainder of the paper, we fix a category $\cO$ datum $(G,\FI,\mass).$
Now
observe that the inclusion $G_c\hookrightarrow D_c$ defines a trihamiltonian action of $G_c$ on $T^*\CC^n,$ whose moment maps may be recovered from those described in \Cref{ex:toric-moment-maps} by composition with the quotient map $\mathfrak{d}_\RR^\vee\to\mathfrak{g}_\RR^\vee.$
\begin{definition}\label{defn:hypertoric-variety}
    The \emph{toric hyperk\"ahler manifold} (or  hypertoric variety) associated to the exact sequence \Cref{eq:basic-exact-sequence} and the regular stability parameter $\FI$ is the hyperk\"ahler reduction
    \[
    \htvar_G(\FI):=\mu^{-1}(\FI,0,0)/G_c.
    \]
\end{definition}
\begin{theorem}[\cite{Biel-Dan}*{Theorem 3.2}]
    The space $\htvar_G(\FI)$ is a complete hyperk\"ahler manifold.
\end{theorem}
The varieties $\htvar_G(\FI)$ were first introduced in \cite{Goto} and subsequently studied in \cites{Konno, Biel-Dan,Haus-Sturm}.
\begin{remark}
    The condition of unimodularity in the definition of a category $\cO$ datum ensures that the absence of any strictly $\FI$-semistable points; without this condition, $\htvar_G(\FI)$ would no longer be guaranteed to be a smooth manifold, but it would still be a smooth orbifold (or DM stack) by \cite{Biel-Dan}*{Theorem 3.2}. Most of the results in this paper would still hold in this case, but we would no longer be able to make a comparison with Fukaya categories.
\end{remark}

\subsection{Distinguished Lagrangians}

Every hyperkähler toric manifold $\htvar_G(\FI)$ admits a distinguished collection of subvarieties which are $I$-holomorphic and Lagrangian with respect to the complex-symplectic form $\Omega_I$. These are the images of the conormals to coordinate strata in $\mathbb{C}^n$. 

For $\alpha \in 2^{[n]}$, let $Z_\alpha\subset \CC^n$ be the toric subvariety defined by 
        \begin{equation*}
        Z_\alpha:=\{(z_1,\ldots,z_n)\in \CC^n\mid z_{i_k}=0 \text{ if } i_k= -\}.
        \end{equation*}

\begin{definition}
We write $\LL^\alpha:=T_{Z_\alpha}^*\CC^n \subset T^*\mathbb{C}^n$ for the conormal to the toric subvariety $Z_\alpha$. 
    We will denote by $\LL_n$ (or just $\LL$ if $n$ is understood) the union of conormals to toric subvarieties:
    $
    \LL_n:=\bigcup_{\alpha\in 2^{[n]}}\LL^\alpha.
    $
\end{definition}

Consider the standard hyperplane arrangement on $\fdv_\RR \simeq \RR^n$ given by the coordinate hyperplanes, cooriented positively. We write $\Delta_\alpha$ for the chamber with sign vector $\alpha \in 2^{[n]}$.

\begin{lemma}\label{lemma:image-moment-lag}
    The chamber $\Delta_\alpha$ is the image of the Lagrangian $\LL^\alpha$ under the real moment map $\mu^\RR$.
\end{lemma}
\begin{proof}
    The restriction of the $D_c$-action, and the real K\"ahler form, to a Lagrangian $\LL_\alpha,$ is the standard torus action on the toric variety $\LL_\alpha\cong \CC^n,$ with some signs reversed as indicated in the sign vector $\alpha$, and the chamber $\Delta_\alpha$ is its moment polytope.
\end{proof}

Now we pass to the $G$ quotient.
\begin{definition}We consider the following Lagrangians obtained from $\LL=\LL_n$:
\begin{itemize}
    \item  We write $\LL_G:=\LL/G\subset (\mu^{\CC})^{-1}(0)/G = T^*(\CC^n/G)$ for the Lagrangian substack of $T^*(\CC^n/G).$
    \item We write $\LL_G(\FI,-):=(\LL\cap\mu^{-1}(\FI,0,0))/G_c\subset \fX_G(\FI).$ Equivalently, we 
    we can write this Lagrangian as the intersection
    $\LL_G\cap\fX_G(\FI),$ taken in $T^*(\CC^n/G)$ using the embedding of \Cref{cor:substack-embedding}.
    \item For $\alpha\in 2^{[n]},$ we write $\LL_G^\alpha(\FI,-):=(\LL^\alpha\cap\mu^{-1}(\FI,0,0))/G_c\subset \fX_G(\FI).$
\end{itemize}
\end{definition}

From \Cref{lemma:image-moment-lag} we can see that $\LL_G^\alpha(\FI, -)$ will be nonempty precisely when $\alpha\in \cF$ is a feasible sign vector, so that $\LL_G(\FI,-)$ may be written
\[
\LL_G(\FI,-)=\bigcup_{\alpha\in \cF}\LL_G^\alpha(\FI,-).
\]
For $\alpha\in\cF,$ we can describe the Lagrangian $\LL^\alpha_G(\FI, -)$ explicitly.
Observe that the manifold $\htvar_G(\FI)$, obtained from $\HH^n$ through hyperk\"ahler reduction by $G_c\subset D_c,$ has a residual trihamiltonian action of $F_c=D_c/G_c,$ from which it has a real moment map
\begin{equation}\label{eq:real-f-momentmap}
\mu^\RR:\htvar_G(\FI)\to \frf^\vee_\RR(\FI)\subset \fdv_\RR,
\end{equation}
whose codomain carries the hyperplane arrangement $\cH(\FI).$
\begin{proposition}\label{prop:moment-polytopes}
    For $\alpha\in \cF,$ the Lagrangian $\LL^\alpha_G(\FI, -)$ is a toric variety, and the restriction of \Cref{eq:real-f-momentmap} is a moment map for its dense $F_c$ action, with image given by the polytope $\Delta_\alpha\cap \hpspace$ in the hyperplane arrangment $\cH(\FI).$
\end{proposition}
\begin{proof}
    The assertions of the proposition all follow from the fact that $\LL^\alpha_G(\FI,-)$ may be identified with the $G_c$-Hamiltonian reduction of the variety $\LL^\alpha\cong \CC^n$ at parameter $\FI,$ compatibly with the $F_c$-action.
\end{proof}


We have not yet taken into account the mass parameter $\mass\in \frf_\RR,$ which we now use to single out certain components of the Lagrangian $\LL_G(\FI,-).$

\begin{definition}
    The {\em category $\cO$ skeleton} is the Lagrangian
    \[
    \Oskel:=\bigcup_{\alpha\in \cF\cap \cB}\LL_G^\alpha(\FI,-)\subset \htvar_G(\FI)
    \]
    given by those components of $\LL_G(\FI,-)$ corresponding to bounded sign vectors $\alpha.$
\end{definition}



Observe that if the mass parameter is integral, it can be understood as a cocharacter $\mass:\CC^\times\to F$ of the torus $F$ which acts on $\htvar_G(\FI),$ determining an action of a 1-dimensional torus $\CC^\times_\mass$ on $\htvar_G(\FI)$; in the non-integral case, it still specifies a complex vector field on $\htvar_G(\FI).$ To simplify some statements below, we will assume the mass parameter is integral, although the generalization to the non-integral case should be clear. The action of the torus $\CC^\times_m$ on $\htvar_G(\FI)$ allows us to give a geometric meaning to the category $\cO$ skeleton:
\begin{lemma}[\cite{Biel-Dan}*{\S 6.5}, \cite{BLPW12}*{Proposition 5.5}]\label{lem:stableset}
    The category $\cO$ skeleton $\Oskel$ is equal to the stable set
    \[\{x \in \htvar_G(\FI) \mid \lim_{z \to 0} z\cdot x \text{ exists} \}\]
    for the action of $z\in \CC^\times_m$. 
\end{lemma}

In \Cref{sec:Lefschetz}, we will explain how the geometry of the $\CC^\times_m$ action may be used to reinterpret $\Oskel$ as the relative skeleton for a Lefschetz fibration.



\begin{remark}
An explicit cocore to $\LL^{\alpha}_G(\FI, \mass)$ was described physically in \cite[\S\S 6.2.3, 6.4.3]{BDGH} and using the BFN construction in \cite[\S 6.6]{HKW}. This cocore is a section of the complex moment map $\mu^\CC:\htvar_G(\FI)\to \frf^\vee_\CC$ that is expected to encode the quantum cohomology of the corresponding component of the 3d mirror Lagrangian skeleton \cite{DGGH, Tel-2d}. A geometric flow relating a quantization of this cocore to a projective module in $\cO^{\dR}(\FI,\mass)$ was described in \cite{Hilburn}.
\end{remark}

For future use, we record a partial ordering which plays a major role in the study of category $\cO.$
\begin{definition}\label{defn:partial-order}
    Let $\alpha,\beta\in \cF\cap\cB\subset 2^{[n]}$ be bounded and feasible sign vectors. By feasibility, they correspond to chambers $\Delta_\alpha,\Delta_\beta$ in the hyperplane arrangement $\cH(\FI).$ By boundedness, the restriction of the functional $m$ to each chamber takes a maximal value on some vertex, which we denote $p_\alpha\in\Delta_\alpha$ and $p_\beta\in \Delta_\beta,$ respectively. 
    Then we define a partial order on $\cF\cap\cB$ by declaring that $\alpha\leq \beta$ precisely when $m(p_\alpha)\leq m(p_\beta).$
\end{definition}

\subsection{Liouville geometry}\label{subsection:liouville-geometry}


An excellent source for the Liouville geometry of symplectic resolutions is \cite{vzivanovic2022exact}. The following is an elaboration of the Remark following Theorem 7 of \cite{shende-hitchin}. 

\begin{lemma}\label{lem:HK-Weinstein}
    Let $\fX$ be a hyperk\"ahler manifold equipped with an $I$-holomorphic action $\CC^\times\curvearrowright\fX$ of $\CC^\times,$ such that for all $x\in\fX,$ the limit $\lim_{\lambda\to 0}\lambda\cdot x$ exists.
    \begin{enumerate}
        \item Suppose the holomorphic symplectic form $\Omega_I$ has positive weight for this action. Then up to rescaling, the vector field generated by the action of $\RR_{>0}\subset \CC^\times$ is a Liouville vector field on $\fX$ (for real symplectic form $\omega_J$).
        \item Suppose moreover that the $S^1\subset \CC^\times$ action is Hamiltonian for the K\"ahler form $\omega_I$. Then the Hamiltonian function underlies a (generalized, i.e., Morse-Bott type) Weinstein structure on $(\fX,\omega_J)$.
    \end{enumerate}
\end{lemma}
\begin{proof}
    The assumption that the symplectic form has positive weight for the symplectic form guarantees that the vector field for the $\RR_{>0}$ action scales it by a positive factor, so that some rescaling $V$ of this vector field satisfies the Liouville condition $\cL_V(\Omega_I) = \Omega_I.$ This establishes (1).
    
For (2), observe that holomorphicity of the $\CC^\times$ action means that the vector field for the $S^1$ action may be written as $I\cdot V,$ where $V$ is the Liouville vector field described above. Using the relation $\omega_I(-,I\cdot V) = g(-,-)$ where $g$ is the hyperk\"ahler metric, we conclude that $V$ agrees with the gradient vector field for a Hamiltonian function for $I\cdot V.$ Since this function is the moment map for a $\CC^\times$-action, it is Morse-Bott.
\end{proof}

To equip $\htvar_G(\FI)$ with a Weinstein structure, it will suffice to construct a $\CC^\times$ action satisfying the hypotheses of \Cref{lem:HK-Weinstein}. To this end: let $\SS\simeq \CC^\times$ be the 1-dimensional torus which acts on $T^*\CC^n\simeq \CC^n\times \CC^n$ by scaling each factor with weight 1. The action of $\SS$ on $T^*\CC^n$ descends to an action
on $\htvar_G(\FI).$

\begin{proposition}\label{prop:Weinstein-vf}
    The action $\SS\curvearrowright \htvar_G(\FI)$ satisfies the hypotheses of \Cref{lem:HK-Weinstein}.
\end{proposition}
\begin{proof}
    The condition that the action of $\SS$ on $\htvar_G(\FI)$ dilates the holomorphic symplectic form, and that $S^1\subset \SS$ is Hamiltonian for the real K\"ahler form, follow from the analogous facts about the action of $\SS$ on $T^*\CC^n.$ 
\end{proof}



Recall that the \emph{skeleton} of a Liouville manifold $\fX$ is the stable set for the Liouville vector field, i.e., the set of points which do not escape to infinity under Liouville flow. If $\fX$ is moreover assumed to be Weinstein, then this is the set of points $x\in \fX$ for which the limit $\lim_{s\to\infty}\phi^s(x)$ under the Liouville flow exists.

\begin{lemma}[\cite{BLPW10}*{Proposition 5.5}]\label{lem:skeleton-is-compacts}
    The skeleton of $\htvar_G(\FI)$ is the union of compact components in
    the extended core $\LL_G(\FI,-).$
\end{lemma}
\begin{proof}
    There is an $\SS$-equivariant map $\nu:\htvar_G(\FI)\to \fX_G(0),$ and the latter is an affine variety with a single fixed point $o$. The skeleton of $\htvar_G(\FI)$ is the preimage $\nu^{-1}(0),$ which is the union of compact components of $\LL_G(\FI,-)$ by \cite{Biel-Dan}*{Theorem 6.5}. 
\end{proof}

\subsection{A Lefschetz fibration}\label{sec:Lefschetz}
In this section, which is not used in the main results of our paper, we record an alternative perspective on the category $\cO$ skeleton $\Oskel$ as the relative skeleton for a Lefschetz fibration on the hypertoric variety $\htvar_G(\FI).$
Recall that we write
$\CC^\times_m$ for the 1-dimensional torus which acts on $\fX_G(\FI)$ through the cocharacter given by the mass parameter $m:\CC^\times\to F.$
\begin{definition}
    We will write $W^\mass_J:\htvar_G(\FI)\to \CC$ for the $J$-holomorphic moment map (for the holomorphic symplectic form $\Omega_J$) for the $\CC^\times_m$-action on $\htvar_G(\FI).$
\end{definition}

\begin{proposition}
    The function $W^\mass_J$ is a Lefschetz fibration. Its critical points are the set of 0-strata in $\LL_\FI.$
\end{proposition}
\begin{proof}
    The critical points of the $\CC^\times_m$-moment map $W^\mass_J$ are the fixed points of the $\CC^\times_m$ action on $\htvar$. By regularity of $m$, these are the same as the fixed points of the torus $F,$ which are precisely the 0-strata in $\LL_\FI.$ 

    Let $p\in \LL_\FI$ be such a critical point, corresponding to a $\dim(F)$-dimensional stratum in $\LL\subset T^*\CC^n,$ of the form $((\CC^\times)^{\dim(G)}\times \{0\}^{n-\dim(G)})\subset \CC^n\subset T^*\CC^n.$ This stratum has a neighborhood $U_p=T^*(\CC^\times)^{\dim(G)}\times \CC^{n-\dim(G)})$; the image of this neighborhood under hyperk\"ahler reduction gives
    a holomorphic Darboux chart 
    $\overline{U}_p\subset \fX_G(\FI)$ for the critical point $p,$
    $\overline{U}_p\simeq T^*V,$ where $V\simeq \CC^{\dim(F)}$ is the $F$-representation with weights $(1,\ldots,1).$ In these coordinates, the map $W^\mass_J$ becomes $(x_1,\ldots,x_{\dim(F)},y_1,\ldots,y_{\dim(F)})\mapsto \sum x_iy_i.$
\end{proof}

For the rest of this section, fix $N\in \RR$ with $N\gg 0.$

\begin{definition}
    We will write $\frF:=(W^\mass_J)^{-1}(N)\subset \htvar_G(\FI)$ for the fiber of this Lefschetz fibration. We will understand it as a Liouville manifold, equipped with the restriction of Liouville structure $\lambda_J.$
\end{definition}

We will want to think of $\frF$ as the fiber of $W^\mass_J$ ``at infinity.'' 
Observe that the real part $\Re W_J^\mass$ of the function $W^\mass_J$ is the real moment map, with respect to symplectic form $\omega_I,$ for the $S^1_m$ action.

\begin{definition}
    Let $C\subset \htvar_G(\FI)$ be the real hypersurface $C:=(\Re W_J^m)^{-1}(N).$ 
\end{definition}
\begin{lemma}\label{lemma:lefschetz-fiber-skeleton}
    The skeleton of the fiber $\frF$ is precisely the intersection
    $\LL_{\FI,\mass}\cap C.$
\end{lemma}
\begin{proof}
    Observe that the cocharacter $m$ determines a Hamiltonian $S^1$-action on the symplectic manifold $\frF,$ and the Hamiltonian reduction $\frF\redu S^1$ is the hypertoric variety associated to the torus $\tilG:=p^{-1}(\CC^\times_m)\subset D,$ where $p:D\to F$ is the projection from \Cref{eq:basic-exact-sequence}, with stability parameter $(\FI,N)$.
By \Cref{lem:skeleton-is-compacts}, the skeleton of $\frF\redu S^1$ is therefore the image under Hamiltonian reduction of the components of $\LL_\FI$ which become compact and nonempty after this Hamiltonian reduction. These are precisely the noncompact $\mass$-bounded chambers in $\LL_G(\FI,-),$ or in other words the noncompact components in $\Oskel.$
\end{proof}

In \Cref{sec:heuristic}, we offer a reinterpretation of the Fukaya-categorical computations in this paper through the perspective of the above Lefschetz fibration. Also see \Cref{figure:drawing-4d} in \Cref{ssec:examples} for an illustration of the relation of the fiber $\frF$ to the category $\cO$ skeleton.

\section{Algebraic Preliminaries}
In this section we collect some algebraic facts that we will use in our proof.
\subsection{Idempotents and recollements}
Let $A$ be a differential graded $\CC$-algebra, and let $e \in A^0$ be an idempotent. The naive quotient algebra $A/AeA$ is not well-behaved from a homotopical point of view. To remedy this, Braun-Chuang-Lazarev defined in \cite{BCL}*{\S 9} a derived quotient algebra
$A/^L(AeA)$ and gave an explicit dg-model of this algebra by adapting the categorical dg-quotient construction from \cite{Drinfeld-dgquotients}. We will take this dg-model as our definition of the derived quotient:

\begin{definition}\label{defn:derived-quotient}
The {\em derived quotient} of $A$ by the ideal $A e A$ is 
the dg-algebra 
\[
A /^L AeA := (A\langle h \rangle/(eh=h=he), dh=e)
\]
obtained by adjoining a degree-$(-1)$ variable $h$, satisfying the relations $eh=h=he,$ whose differential is equal to the idempotent $e.$
\end{definition}

\begin{proposition}[\protect{\cite[Remark 9.5]{BCL}, \cite[Section 2.1]{hugel2017recollements}}]\label{proposition:ff-recollement}
There is a recollement diagram
\begin{equation}\label{equation:recollement}
\begin{tikzcd}[column sep=1.5cm]
\Mod_{(A/^L A e A)}  &  
\Mod_A  \arrow[bend right, swap]{l}{i^*}
\arrow[from=l, hook]{}{i_*=i_!}
\arrow[bend left]{l}{i^!} & 
\arrow[bend right, swap]{l}{j_!}
\arrow[from=l, "j^!=j^*"]
\arrow[bend left]{l}{j_*}  \Mod_{eAe}.
\end{tikzcd}
\end{equation}
of stable categories where 
\[
\begin{array}{ll}
i^*:= (-) \otimes_A A/^L AeA,&
j_! := (-) \otimes_{eAe}eA,\\
i_*:= \hom_{A/^L AeA}(A/^L AeA,-),&
j^! = \hom_A(eA,-),\\
i_!:= (-)\otimes_{A/^L AeA} A/^L AeA,&
j_*:= \hom_{eAe}(Ae,-),\\
i^! = \hom_A(A/^L AeA, -),&
j_*:= \hom_{eAe}(Ae,-).
\end{array}
\]
In other words, the following properties hold:
\begin{itemize}
    \item[(i)] $(i^*, i_*, i^!)$ and $(j_!, j^*, j_*)$ are adjoint triples,
    \item[(ii)] $i_*=i_!, j_*, j_!$ are fully faithful,
    \item[(iii)] $j^* \circ i_*=0$ (hence by adjointness $i^*j_!=0$ and $i^!j_*=0$),
    \item[(iv)] for all $X \in \Mod_A$, the unit can be completed to an exact triangle
    \begin{equation*}
    i_*i^! X \to X \to j_*j^*X \to i_*i^! X[1]
    \end{equation*}
    and the counit can be completed to an exact triangle
    \begin{equation*}
    j_!j^* X \to X \to i_*i^*X \to j_!j^* X[1].
    \end{equation*}
\end{itemize}
\end{proposition}


When $A$ is an ordinary algebra, the recollement \eqref{equation:recollement} appeared in the work of Cline, Parshall, and Scott \cite{CPS}. Rather than using the derived quotient, they assumed an extra hypothesis which guarantees that the derived quotient agrees with the naive one.
\begin{definition}\label{defn:ff-recollement}
Let $A$ be an algebra and let $e \in A$ be an idempotent. The ideal $AeA$ is \emph{stratifying} if any of the following equivalent conditions holds:
\begin{enumerate}
\item $Ae \otimes_{eAe} eA  \cong H^0(Ae \otimes_{eAe} eA) = AeA $.
\item The functor $\tilde{i}_!: \Mod_{A/AeA} \to \Mod_A$ defined by $\tilde{i}_! := (-)\otimes_{A/ AeA} A/ AeA$ is fully faithful.
\item $A/^L AeA \cong H^0(A/^L AeA) = A/AeA$.
\end{enumerate}
When this is the case, we will not distinguish between $i_!$ and $\tilde{i}_!$.
\end{definition}

\begin{remark}
There is also a notion of a recollement diagram for abelian categories. When $A$ is an ordinary algebra, the diagram obtained from \eqref{equation:recollement} by taking hearts of the standard t-structures is always a recollement of \emph{abelian} categories, regardless of whether or not $AeA$ is stratifying.
\end{remark}

\subsection{Koszul duality}\label{ssec:koszul-abstract}
Unlike \cite{BGS, BLPW10}, which consider mixed Koszul duality between ordinary graded algebras, we will be interested in derived Koszul duality between differential graded algebras or, more generally, associative ring spectra \cite{DGI} \cite[\S 14]{Lurie-SAG}.


\begin{remark}
In this section we will need to distinguish between left and right modules. Our convention is that endomorphisms of a right module act on the left and vice versa. 
\end{remark}

Let $S = (\CC)^{\times n}$. An \emph{augmented differential graded $S$-ring} is a differential graded $k$-algebra $A$ equipped with $\CC$-algebra homomorphisms $S \to A$ and $A \to S$ such that the composition $S \to A \to S$ is the identity. The canonical example of an augmented dg $S$-ring is the \emph{tensor algebra} $T_S(V) = \oplus_{n \geq 0} V^{\otimes_S n}$ of a finite dimensional differential graded $S$-bi-module $V$. {\bf Note that $S$ is not required to be central in $A$.} A reference where standard results on Koszul duality for augmented dg $k$-algebras are extended to dg $S$-rings is \cite{Han-hochschild}.

Suppose that we have two augmented dg $S$-rings $A$ and $B^{\op}$ such that there exists an augmentation $A \otimes_{\CC} B^{\op} \to S$ of the tensor product algebra extending the augmentations on $A$ and $B^{\op}$. This is equivalent to giving maps
\begin{align}\label{eq:BKDA}
B &\to \hom_{A}(S,S),    \\
\label{eq:AKDB}
A &\to \hom_B(S, S),
\end{align}
\noindent where we are considering $S$ as a right $A$-module and a left $B$-module.

\begin{definition}\label{def:koszul-dual}
We say that {\em $B$ is Koszul dual to $A$} if \Cref{eq:BKDA} is an isomorphism. We say that {\em $A$ is Koszul dual to $B$} if \Cref{eq:AKDB} is an isomorphism. When 
both of these conditions are satisfied, we will say that $A$ and $B$ are \emph{Koszul bidual}.
\end{definition}

\begin{definition}
A differential graded algebra $A$ is \emph{connective} if $H^k(A)=0$ for all $k>0$. It is \emph{locally finite} if $H^k(A)$ is finite dimensional for all $k$ and \emph{finite} if $H^*(A)$ is finite dimensional.
\end{definition}

\begin{proposition}[\protect{\cite[Theorem 4.2.8, \S 8.1]{Booth-deformation}}]\label{prop:booth-involution}
Suppose that that $A$ is connective, locally finite, and that the augmentation ideal of $H^0(A) \to S$ is nilpotent. Then $A \simeq \hom_{\hom_A(S,S)}(S,S)$. In particular, Koszul duality is an involution.
\end{proposition}

\begin{remark}
For a different finiteness condition that also ensures that Koszul duality implies Koszul biduality see \cite[\S 14.1.3]{Lurie-SAG}.   
\end{remark}

\begin{remark}\label{rem:doubledual}
In general, one expects the double Koszul dual of $A$ to be the derived completion of $A$ at the augmentation ideal, as explained in \cites{Porta-completion, Efimov-completion}. For example, the Koszul dual of $\CC[x]$ with $|x|=0$ is $\CC[\epsilon]$ with $|\epsilon|=1$, 
and taking the Koszul dual again gives $\CC[[x]]$. 
One way to avoid this difficulty is to impose the finiteness hypotheses of \Cref{prop:booth-involution}; another approach is to work in the mixed setting,
since the completion of $\CC[x]$ as a graded algebra is equivalent to $\CC[x]$. By applying the degrading functor to the mixed equivalence one gets \cite[\S 2]{Chen-cyclic}.
\end{remark}

The augmentation $A \otimes_{\CC} B^{op} \to S$ gives rise to contravariant functors
\begin{equation} \label{eq:kdadj}
\begin{tikzcd}[row sep = .25cm, column sep = 3cm]
\RMod_{A} \arrow[yshift=0.9ex]{r}{\hom_A(-,S)}
\arrow[leftarrow, yshift=-0.9ex]{r}[yshift=-0.2ex]{}[swap]{\hom_{B}(-, S)} & \LMod_{B} 
\end{tikzcd}    
\end{equation}
which are adjoint on the right, in the sense that there are natural equivalences
\begin{equation}
\hom_A(M,\hom_B(N,S)) \cong \hom_{A \otimes_{\CC} B^{op}}(M \otimes_{\CC} N, S) \cong \hom_B(N,\hom_A(M,S)).
\end{equation}
\begin{remark}
To speak of the left or right adjoint of a contravariant functor $F: \cC \to \cD$ one needs to decide whether to view it as a covariant functor $\cC^{\op} \to \cD$ or $\cC \to \cD^{\op}$. A left adjoint from one perspective is a right adjoint from the other. On the other hand, saying that two contravariant functors are adjoint on the left or adjoint on the right is unambiguous.
\end{remark}

When $A$ is Koszul dual to $B$ or vice versa, we will see that the Koszul duality functors \Cref{eq:kdadj} restrict to equivalences between certain subcategories of $\RMod_A$ and $\LMod_B$.
 
\begin{definition}
A subcategory of a stable category $\cC$ is \emph{thick} if it is a stable subcategory that is closed under retracts. For $X$ an object of $\cC,$ we write $\cT_{\cC}(X)$ for the smallest thick subcategory of $\cC$ containing $X.$ 
\end{definition}

\begin{remark}
The category $\cT_{\RMod_A}(A)$ (resp., $\cT_{\LMod_B}(B)$) is also known as the category of perfect modules and denoted by $\RPerf_A$ (resp. $\LPerf_B$).
\end{remark}

The following proposition is proved under the assumption of Koszul biduality in \cite{Campbell-Koszul}. But by examining the proof, it is easy to see that one can split the result. Also see \cite[Section 14.6]{Lurie-SAG} for a similar result after ind-completion.

\begin{proposition}[\protect{\cite[Theorem 2.43]{Campbell-Koszul}}]
If $B$ is Koszul dual to $A$ (resp., $A$ is Koszul dual to $B$) the functors \Cref{eq:kdadj} restrict to an equivalence $\cT_{\RMod_A}(S) \simeq \LPerf_B$ (resp., $\RPerf_A \simeq \cT_{\LMod_B}(S)$).
\end{proposition}

Now suppose that $e \in S$ is an idempotent and $e^c = 1-e$ is the complementary idempotent. Let $T = e^c S e^c \simeq S/SeS$. The derived quotient $A/^L (A e A)$ and the cornering\footnote{This terminology is taken from \cite{CIK18} by way of \cite{Booth-deformation}.}
$(e^c B e^c)^{\op}$ are both $T$-augmented. Moreover, the augmentation $A \otimes_S B^{\op} \to S$ induces an augmentation $A/^L(A e A) \otimes_T (e^c B e^c)^{\op} \to T$, so one can ask whether $A/^L (A e A)$ and $e^c B e^c$ are Koszul dual.

First, recall the recollement \Cref{equation:recollement} associated to $e \in A$. There is an analogous recollement associated to $e^c \in B$ with functors $(i^c)^* \adj (i^c)_* = (i^c)_! \adj (i^c)^!$ and $(j^c)_! \adj (j^c)^! = (j^c)^* \adj (j^c)_*$.

\begin{proposition} \label{prop:derivedKDID}
\begin{enumerate}
\item
Suppose $B$ is Koszul dual to $A$. Then $e^c B e^c$ is Koszul dual to $A/^L(AeA)$ and the following diagram commutes:
\begin{equation}\label{eq:kdidcomm1}
\begin{tikzcd}[column sep = 4cm]
\cT_{\RMod_A}(S) \arrow[r, "{\hom_A(-,S)}"] & (\LPerf_B)^\op \\
\cT_{\RMod_{A/^L(A e A)}}(T) \arrow[r, "{\hom_{A/^L(A e A)}(-,T)}"] \arrow[u, "i_*"] & (\LPerf_{e^c B e^c})^\op \arrow[u, swap, "((j^c)_!)^\op"].
\end{tikzcd}
\end{equation}
\item
Suppose $A$ is Koszul dual to $B$. Then $A/^L(AeA)$ is Koszul dual to $e^c B e^c$ if and only if the following diagram commutes:
\begin{equation}\label{eq:kdidcomm2}
\begin{tikzcd}[column sep = 4cm]
\RPerf_A \arrow[d, "i^*"] & (\cT_{\LMod_B}(S))^\op \arrow[l, swap, "{\hom_B(-,S)}"] \arrow[d, "((j^c)^!)^\op"]  \\
\RPerf_{A/^L(A e A)}  & (\cT_{\LMod_{e^c B e^c}}(T))^\op  \arrow[l, swap, "{\hom_{e^c B e^c}(-,T)}"].
\end{tikzcd}
\end{equation}
\end{enumerate}
\end{proposition}

\begin{proof}
\begin{enumerate}
\item It is sufficient to check commutativity on generators. First, notice that
\[
\hom_A(i_*(T), S) \simeq \hom_A(e^cS, S) \simeq \hom_A(S,S) e^c \simeq B e^c \simeq (j^c)_!(e^c B e^c).
\]
Thus, \Cref{eq:kdidcomm1} commutes if and only if $e^c B e^c$ is Koszul dual to $A/^L(AeA)$. But since $i_*$ is fully faithful, we have
\[
\hom_{A/^L(A e A)}(T,T) \simeq \hom_A(i_*(T), i_*(T)) \simeq \hom_B(B e^c, B e^c) \simeq e^c B e^c.
\]

\item Diagram \Cref{eq:kdidcomm2} commutes if and only if
\[
A/^L(AeA) \simeq i^*(A) \simeq i^*\hom_B(S,S) \simeq \hom_{e^c B e^c}((j^c)^!(S), T) \simeq \hom_{e^c B e^c}(T,T).\qedhere
\]
\end{enumerate}
\end{proof}

\begin{corollary} \label{cor:kdid}
Suppose $B$ is Koszul dual to $A$ and $e \in S$ is an idempotent. Assume $A /^L( A e A)$ is connective, locally finite, and that the augmentation ideal of $H^0(A/^L (AeA)) \to T$ is nilpotent. In this case, $A /^L( A e A)$ and $e^c B e^c$ are Koszul bidual.
\end{corollary}

\begin{corollary}\label{cor:kdidcomm}
Suppose $A$ and $B$ are Koszul bidual and $e \in S$ is an idempotent. If the diagram 
\begin{equation}\label{eq:kdidcomm3}
\begin{tikzcd}[column sep = 4cm]
\RMod_A \arrow[r, "{\hom_A(-,S)}"] & (\LMod_B)^\op \\
\RMod_{A/^L(A e A)} \arrow[r, "{\hom_{A/^L(A e A)}(-,T)}"] \arrow[u, "i_*"] & (\LMod_{e^c B e^c})^\op \arrow[u, swap, "((j^c)_!)^\op"]
\end{tikzcd}
\end{equation}
commutes, then the diferential graded algebras $A/^L(AeA)$ and $e^c B e^c$ are Koszul bidual.
\end{corollary}

\begin{proof}
If \Cref{eq:kdidcomm3} commutes then so does the diagram obtained by passing to left adjoints. The resulting diagram restricts to \Cref{eq:kdidcomm2}. Note that $(-)^\op$ swaps left and right adjoints.
\end{proof}

\begin{remark}
Under certain finiteness conditions, the converse to \Cref{cor:kdidcomm} is true in the mixed setting \cite[Theorem 8.23]{BLPW12}. In the derived setting it is sufficient to assume $\RPerf_A \simeq \cT_{\RMod_A}(S)$ and $\RPerf_{A/^L(AeA)} \simeq \cT_{\RMod_{A/^L(AeA)}}(T)$. Is this necessary?
\end{remark}

Now assume that $A = A^0$ is concentrated in cohomological degree $0$ and that it is equipped with an additional grading $A_{\bullet}$ (conventionally referred to as the {\em mixed grading} to distinguish it from the cohomological grading) such that the augmentation induces an isomorphism $A_0 \simeq S$. We will also assume that $A$ is \emph{mixed co-connective} and \emph{mixed locally finite}, i.e., that $A_i=0$ for $i<0$ and $A_i$ is finite dimensional for all $i \geq 0$. Priddy discovered a class of such graded algebras whose Koszul duals are especially easy to describe \cite{Priddy}. These were later studied by Beilinson, Ginzburg, and Soergel \cite{BGS}.

\begin{definition}
An $S$-augmented graded algebra $A$ is \emph{Koszul} if it satisfies the graded Tor-vanishing condition $H^i(S \otimes_A S)_j = 0$ for $i \neq -j$.
\end{definition}

\begin{definition}
The $S$-augmented graded algebra $A$ is \emph{quadratic} if there exists a finite dimensional $S$-bi-module $V$ and a subspace $R \subseteq V \otimes V$ such that $A \simeq T_S(V/\langle R \rangle$. The \emph{quadratic dual} of $A$ is the graded algebra $Q(A) = T_S(V^\vee)/ \langle R^\perp \rangle$. Here both $V$ and $V^\vee$ are in mixed degree $1$.
\end{definition}

\begin{remark}
Note that $Q(A)$ is also mixed locally finite and mixed co-connective.
\end{remark}

\begin{proposition}[\protect{\cite[Proposition 1.2.3]{BGS}}]
Every Koszul algebra is quadratic.    
\end{proposition}

\begin{proposition}[\protect{\cite[Theorem 1.2.5]{BGS}}]
Quadratic duality is involutive, i.e., $A \simeq Q(Q(A))$. Moreover, $Q(A)$ is Koszul if and only $A$ is.
\end{proposition}

The following proposition compares mixed Koszul duality of graded algebras as studied in \cite{BGS} with Koszul duality of differential graded algebras.

\begin{proposition}[\protect{\cite[Theorem 7.12]{Brantner}}]\label{prop:KDQD} \label{prop:QDKD}
If $A$ is Koszul then its Koszul dual is the dg-algebra $A^! = (Q(A)^\op, 0)$, where the mixed grading on $Q(A)^\op$ becomes the homological grading on $A^!$. In particular, we see that $\hom_A(S,S)$ is formal.
\end{proposition}

\begin{remark}
The Koszul dual of the algebra $Q(A)^\op$ is the dga $(A,0)$ but the Koszul dual of the dga $A^!=(Q(A)^\op, 0)$ is only the algebra $A$ when $A$ is finite dimensional. In general the Koszul dual of $A^!$ is the derived completion as described in \Cref{rem:doubledual}.
\end{remark}

Now we will use Koszul duality to give a criterion for proving that $A/AeA \cong A/^L(AeA)$. The key observation, \Cref{prop:KDId}, is a slight variant of \cite[Proposition 8.23]{BLPW12}.

\begin{lemma}[\protect{\cite[Lemma 8.22]{BLPW12}}] \label{lem:QDID}
Suppose that $A$ is a quadratic algebra and $e \in A$ is an idempotent such that $e^c Q(A) e^c \subseteq Q(A)$ is quadratic as a $T = e^c S e^c$-augmented algebra. Then $Q(A/(AeA)) \simeq e^c Q(A) e^c$. In particular, if $A/AeA$ is Koszul as a $T$-augmented algebra its Koszul dual is $e^c A^! e^c$.
\end{lemma}

\begin{proposition}\label{prop:KDId}
Suppose that $A$ is Koszul and $e \in A$ is an idempotent such that $e^c Q(A) e^c$ is quadratic and $A/(AeA)$ is Koszul. Then there is a commutative diagram
\begin{equation} \label{eq:kdidcomm4}
\begin{tikzcd}[column sep = 4cm]
\cT_{\RMod_A}(S) \arrow[r, "{\hom_A(-,S)}"] & (\LPerf_{A^!})^\op \\
\cT_{\RMod_{A/(A e A)}}(T) \arrow[r, "{\hom_{A/(A e A)}(-,T)}"] \arrow[u, "i_*"] & (\LPerf_{e^c A^! e^c})^\op \arrow[u, swap, "((j^c)_!)^\op"].
\end{tikzcd} 
\end{equation}
In particular, $i_*$ is fully faithful since $(j^c)_!$ is.
\end{proposition}

\begin{proof}
The same argument that proved commutativity of \Cref{eq:kdidcomm1} (with $B=A^!$ and $A/(AeA)$ in place of $A/^L(AeA)$) proves commutativity of \Cref{eq:kdidcomm4}. In particular, we know that $i_*(T)= e^c S$ and we know that $\hom_{A/(AeA)}(T,T) \simeq e^c A^! e^c$ by \Cref{lem:QDID}.
\end{proof}

\begin{corollary}\label{proposition:koszul-recollement}
Suppose that $A$ is Koszul and $e \in A$ is an idempotent such that $e^c Q(A) e^c$ is quadratic and $A/(AeA)$ is Koszul. Moreover, assume that $A/(AeA) \in \cT_{\RMod_{A/(AeA)}}(T)$ and $S \in \RPerf_A$. Then $A/(AeA) \simeq A/^L(AeA)$.
\end{corollary}

\begin{proof}
By \Cref{defn:ff-recollement} it suffices to show that $i_*: \RMod_{A/(AeA)} \to \RMod_A$ is fully faithful. By our assumptions $\RPerf_{A/(AeA)} \subseteq \cT_{\RMod_{A/(A e A)}}(T)$ and $\cT_{\RMod_A}(S) \subseteq \RPerf(A)$. Thus we see that $i_*$ is the ind-completion of the fully faithful embedding
\[
\RPerf_{A/(AeA)} \subseteq \cT_{\RMod_{A/(A e A)}}(T) \xhookrightarrow{i_*} \cT_{\RMod_A}(S) \subseteq \RPerf(A).\qedhere
\]
\end{proof}


\subsection{Highest weight categories and semi-orthogonal decompositions}\label{subsection:highest-weight}

Let $\cC$ be an abelian category. An object is \emph{simple} if it has no subobjects. We say that $\cC$ is Artinian (resp.\ Noetherian) if any descending (resp.\ increasing) sequence of subobjects stabilizes. Given an object $X \in \cC$, a \emph{composition series} is a finite sequence of subobjects $0 \to Y_1 \hookrightarrow \dots \hookrightarrow Y_k \hookrightarrow X$ such that each subquotient is simple and nonzero. 

It is straightforward to verify that an abelian category is Artinian and Noetherian iff every object admits a composition series. In that case, the Jordan--Hölder lemma implies that the composition series is essentially unique. 

We now recall the notion of a highest weight category,
following \cite{BLPW16}. For more details, we refer to \cite{CPS} (noting that the definition there is slightly more general than that used in \cite{BLPW16}).

\begin{definition}[\cite{BLPW16}*{Definition 4.11}] Let $\cC$ be a $\mathbb{C}$-linear category which is abelian and artinian. Let $\{S_\alpha\}_{\alpha \in \cI}$ and $\{P_\alpha\}_{\alpha \in \cI}$ be an enumeration of the simples and their projective covers and fix a partial order on $\cI$. We say that $\cC$ is \emph{heighest weight} with respect to the poset $\cI$ if there is a collection of objects $\{V_\alpha\}_{\alpha \in \cI}$ and epimorphisms $\{ P_\alpha \xrightarrow{\Pi_\alpha} V_\alpha \xrightarrow{\pi_\alpha} S_\alpha\}$
so that for each $\alpha,$
\begin{itemize}
    \item  $\ker \pi_\alpha$ has a finite filtration whose subquotients are isomorphic to $S_\beta$ for some $\beta <\alpha,$ and
    \item $\ker \Pi_\alpha$ has a finite filtration whose subquotients are isomorphic to $V_\gamma$ for some $\gamma>\alpha.$
\end{itemize}
The $V_\alpha$ are called \emph{standard objects}.
\end{definition}

\begin{definition}
   An algebra is \emph{quasi-hereditary} if its category of finitely generated right modules is highest weight, for some ordering on the simple modules. 
\end{definition}


\begin{remark}
    Let $\cC$ be a highest-weight category. Then 
    the $\{S_\alpha\}, \{V_\alpha\}, \{P_\alpha\}$ each form a basis for the Grothendieck group $K_0(\mathcal{C})$; dually, so do the collections of indecomposable injectives $\{I_\alpha\}$, costandards $\{\Lambda_\alpha\}$, and tilting objects $\{T_\alpha\}.$ 
    Moreover, these objects satisfy {\em BGG reciprocity}: the multiplicity of $V_\alpha$ in $P_\beta$ coincides with the multiplicity of $S_\beta$ in $V_\alpha.$
\end{remark}

%

The feature of a highest-weight category $\cC$ most salient to us here is the presence of a full exceptional collection. As shown in \cite{krause2017highest}, this is essentially equivalent to the highest weight structure.

\begin{theorem}[\cites{Dlab-Ringel-moduletheoretic,krause2017highest}]
    Let $\cC$ be a highest weight category. Then the standard objects $V_\alpha$ form a full exceptional collection: $\operatorname{Ext}^*(V_\alpha)=\CC\cdot \id_{V_\alpha}$ for all $\alpha,$ and for $i>0,$ $\operatorname{Ext}^i(V_\alpha,V_\beta) =0$ unless $\alpha<\beta.$
\end{theorem}

\begin{remark}
    To a symplectic geometer, the above should indicate that quasi-hereditary algebras naturally arise as the endomorphism algebras of cocores in the Fukaya--Seidel category of a Lefschetz fibration, and that the structural properties their module categories enjoy admit geometric explanations, with Lefschetz thimbles playing the role of standard objects. It is indeed possible to understand category $\cO$ in these terms; we will explain a heuristic dictionary in \Cref{sec:heuristic}.
\end{remark}

\subsection{Betti and de Rham $G$-actions}
In geometric representation theory it is common to consider several different notions of an action of a reductive group $G$ on a category: any covariantly functorial sheaf theory ${\sf SHV}$ satisfying the Ku\"nneth formula gives rise to a convolution monoidal structure
\begin{align*}
&{\sf SHV}(G) \otimes {\sf SHV}(G) \simeq {\sf SHV}(G \times G) \xrightarrow[]{m_*} {\sf SHV}(G), \\
&\Vect \simeq {\sf SHV}(1) \xrightarrow{u_*} {\sf SHV}(G),
\end{align*}
and one can consider module categories for ${\sf SHV}(G)$.\footnote{For nice enough sheaf theories, there is a duality ${\sf SHV}(G)^{\vee} \simeq {\sf SHV}(G)$ identifying $(m_*)^{\vee}$ with a contravariant functoriality $m^!$, giving an identification between $({\sf SHV}(G), m_*)$-modules and $({\sf SHV}(G), m^!)$-comodules.} In this section we will collect some well-known facts about $G$-actions surveyed in \cites{D22-informal} and facts about the Mellin transform from \cites{Tel-ICM, Gannon}.

\begin{definition}\label{def:Gcats}
    A {\em Betti} (resp., {\em de Rham}, {\em weak}) $G$-category is a module category in $\PrL$ for the monoidal category $(\Loc(G), m_*)$ (resp., $(\Dmod(G), m_*),$ $(\QCoh(G),m_*)$).
\end{definition}

\begin{remark}
There is a monoidal functor $\QCoh(G) \to \Dmod(G)$ given by pushforward along the homomorphism $G \to G_{dR}$. As a result, every de Rham $G$-category (sometimes also called a {\em strong} $G$-category) is a weak $G$-category in a canonical way.
\end{remark}

\begin{definition}
Let $\cC$ be a Betti (resp. de Rham, weak) G-category. Then the 
\emph{$G$-invariants} $\cC^{G_\bullet},$ and \emph{$G$-coinvariants} $\cC_{G_\bullet},$
for $\bullet\in \{\Betti,\dR,w\},$
are the respective categories
\[
\cC^{G_\bullet}:=\Hom_{\Mod_{\Shv(G)}}(\Vect,\cC), \qquad \cC_{G_\bullet}:=\Vect\otimes_{\Shv(G)}\cC,
\]
where $\Vect$ is the trivial representation and $\Shv\in\{\Loc, \Dmod, \QCoh\}$ is the appropriate sheaf theory.
\end{definition}

\begin{proposition}
Let $\Shv\in \{\Loc,\Dmod,\QCoh\}.$ Then the trivial $\Shv(G)$-module category $\Vect$ is self-dual. As a result, for $\bullet\in\{\Bet,\dR,w\}$ and $\cC$ a $\bullet$ $G$-category, there is an equivalence $\cC^{G_\bullet}\simeq \cC_{G_\bullet}$ between $G$-invariants and coinvariants.
\end{proposition}

We now specialize to the abelian setting:
\begin{notation}
    For the remainder of this section, we will restrict to the case that $G=(\CC^\times)^k$ is a torus. We will write $G^L:=\Spec(\CC[\pi_1 G])$ for the (Langlands) dual torus.
\end{notation}
In this case, we can identify $G$-categories in terms of linear structure on the dual torus. In the de Rham setting, this equivalence is known as the geometric Mellin transform: see \cite{BZN-Betti}*{\S 2.1}, or \cite{Gannon}*{Appendix A} for a derived enhancement. In the Betti setting, the statement is simpler: see \cite{Tel-ICM}*{Proposition 4.1}.
\begin{proposition}
There are monoidal equivalences
\begin{align*}
(\Loc(G), m_*) &\simeq  (\IndCoh(G^L), \overset{!}{\otimes}) \simeq (\Qcoh(G^L), \otimes), \\
(\Dmod(G), m_*) &\simeq (\IndCoh(\fg^L / \pi_1(G^L)), \overset{!}{\otimes}) \simeq (\Qcoh(\fg^L / \pi_1(G^L)), \otimes)
\end{align*}
that are functorial for surjective homomorphisms $f: G \to H,$ in the sense that the following diagrams commute:
\[
\begin{tikzcd}[row sep = .5cm]
\Loc(G) \arrow[r, "f_*"] \isoarrow{d} & \Loc(H) \isoarrow{d} & \Dmod(G) \arrow[r, "f_*"] \isoarrow{d} & \Dmod(H) \isoarrow{d} \\
\IndCoh(G^L) \arrow[r, "(f^L)^!"] & \IndCoh(H^L) & \IndCoh(\fg^L/\pi_1(G)) \arrow[r, "(Df^L)^!"] & \IndCoh(\mathfrak{h}^L/\pi_1(H^L)).
\end{tikzcd}.
\]
\end{proposition}

\begin{remark}
The exponential map gives an isomorphism between the analytification of the stack $\fg^L/\pi_1(G^L)$ and $G^L$. This is expected to be compatible with the Riemann--Hilbert correspondence.
\end{remark}

\begin{corollary}
A Betti G-action on $\cC$ is equivalent to a $\CC[G^L]$-linear structure on $\cC$.
\end{corollary}

As homomorphism $p: G \to 1$ gives rise to the trivial $G$-representation, the operation of 
taking Betti (resp., de Rham) $G$-invariants coincides with pulling back along $p^L: 1 \to G^L$ (resp., $Dp^L: 0 = \pi_1(G^L)/\pi_1(G^L)  \to \fg^L/\pi_1(G^L)$):
\begin{corollary}
For $\cC$ a Betti (resp. de Rham) $G$-category, there are equivalences
\[
\begin{tikzcd}
\cC^{G_{\Betti}} \simeq \IndCoh(\{1\}) \otimes_{\IndCoh(G^L)} \cC & \cC^{G_{\dR}} \simeq \IndCoh(\{0\}) \otimes_{\IndCoh(\fg^L/\pi_1(G^L))} \cC.
\end{tikzcd}
\]
\end{corollary}
We can similarly use the linear structure to describe weak invariants as the pullback
along ${q:\fg^L \to \fg^L/\pi_1(G^L)}$:
\begin{corollary}
For $\cC$ a weak $G$-category, there is an equivalence
\[
\cC^{G_w} \simeq \IndCoh(\fg^L) \otimes_{\IndCoh(\fg^L/\pi_1(G^L))} \cC.
\]
\end{corollary}

\begin{remark}\label{rem:HC-bimod}
The category $\cC^{G_w}$ is equipped with compatible actions of $(\IndCoh(\fg^L), \overset{!}{\otimes})$ and $(\IndCoh(\pi_1(G^L)), m_*)$. These combine into an action of the monoidal category of Harish-Chandra bimodules
\[
({\sf HC}(G), \star) = (\Dmod(G)^{(G \times G)_w}, m_*).
\]
An ${\sf HC}(G)$-action on $\cC^w$ is equivalent data to a de Rham $G$-action on $\cC.$
\end{remark}
In other words, passing from a de Rham $G$-category $\cC$ to its weak invariants does not lose any data, so long as we remember both the $\IndCoh(\fg^L)$- and $\IndCoh(\pi_1(G^L))$-actions. We will sometimes find it useful to forget part of this data:

\begin{definition}
A {\em $\fg$-category} is a module category for $(\IndCoh(\fg^L), \overset{!}{\otimes})$, or equivalently a category equipped with $(\text{Sym}(\fg) = \CC[\fg^L])$-linear structure.
The {\em $\fg$-invariants} of a $\fg$-category $\cC$ are defined by the pullback
\[
\cC^{\fg}:=\IndCoh(\{0\})\otimes_{\IndCoh(\fg^L)}\cC.
\]
\end{definition}

We will also be interested in imposing a condition slightly weaker than $G$-invariance, namely $G$-monodromicity; as above, this is most easily defined in terms of linear structure for the dual group. Let $(G^L)^\wedge_1$ (resp., $(\fg/\pi_1(G^L))^\wedge_0$) denote the formal completion of $G^L$ (resp., $\fg/\pi_1(G^L)$) along $1$ (resp., $0$). 
\begin{definition}
Let $\cC$ be a Betti (resp. de Rham) $G$-category. The Betti (resp. de Rham) {\em $G$-monodromic category} is
\begin{align*}
\cC^{G_{\Betti}\mon} &:= \IndCoh((G^L)^\wedge_1) \otimes_{\IndCoh(G^L)} \cC,\\
\cC^{G_{\dR}\mon} &:= 
\IndCoh((\fg^L/\pi_1(G^L))^\wedge_0) \otimes_{\IndCoh(\fg/\pi_1(G^L))} \cC.
\end{align*}
\end{definition}

\begin{proposition}
Consider $0 \to \fg^L \to \fg^L/\pi_1(G^L)$. There is an isomorphism $(\fg^L)^\wedge_0 \simeq (\fg / \pi_1(G))^\wedge_0$. Thus, for a de Rham $G$-category $\cC,$ there are equivalences 
\begin{align*}
\cC^{G_{\dR}} &\simeq (\cC^{G, w})^{\fg} = \IndCoh(\{0\}) \otimes_{\IndCoh(\fg)}\cC^{G,w}  \\
\cC^{G_{\dR}\mon} &\simeq (\cC^{G, w})^{\fg\mon} = \IndCoh((\fg^L)^\wedge_0) \otimes_{\IndCoh(\fg)}\cC^{G,w}.
\end{align*}
\end{proposition}

The following proposition was proved for formal completions of schemes in \cite[\S 7.4]{GR-ind}. The case of $\fg^L/\pi_1(G^L)$ follows by examining the cover $q: \fg^L \to \fg^L/\pi_1(G^L)$.
\begin{proposition}\label{prop:completion}
Let $Y \to X$ be one of $1 \to G^L$, $0 \to \fg$, or $0 \to \fg/\pi_1(G^L)$ and let $U = X - Y$. There is an exact sequence
\[
\IndCoh(X^\wedge_Y) \to \IndCoh(X) \to \IndCoh(U)
\]
of $\IndCoh(X)$-module categories. Moreover, we have that
\[
\IndCoh(X^\wedge_Y) = \IndCoh(\varinjlim_k N_k(Y,X)) \simeq \varinjlim_{k} \IndCoh(N_k(Y,X)),
\]
where $N_k(Y,X)$ is the $k$th order neighborhood of $Y$ in $X$.
\end{proposition}

\begin{remark}
The image of the embedding $\IndCoh(X^\wedge_Y) \to \IndCoh(X)$ is the category $\IndCoh(X)_Y$ of ind-coherent sheaves on $X$ which are set-theoretically supported on $Y$.  
\end{remark}

\begin{corollary} \label{cor:monfact}
Let $\cC$ be a Betti G-category. Then
\[
\cC^{G_{\Bet}\mon} \simeq \varinjlim_k \IndCoh(N_k(1,G^L)) \otimes_{\IndCoh(G)} \cC.
\]
Moreover, if $\cC$ is dualizable as an $\IndCoh(G^L)$-module then one has an exact sequence
\[
\cC^{G_{\Bet}\mon} \to \cC \to \IndCoh(G^L-\{1\}) \otimes_{\IndCoh(G)} \cC
\]
in $\PrL$. In particular, the functor $\cC^{G_{\Bet}\mon} \to \cC$ is fully faithful. 

The analogous statements hold for de Rham $G$-categories and $\fg$-categories.
\end{corollary}

To apply the previous corollary, we will need the following fact.
\begin{proposition}[\protect{\cite[Corollary 8.6.3]{GR-study}}]
Let $(\cM, \otimes)$ be a monoidal category in $\PrL$ and let $A$ be an algebra object in $\cM$. The dual of the $\cM$-module category $\RMod_A$ of right $A$-modules is the category $\LMod_{A}$ of left $A$-modules.
\end{proposition}

\section{The de Rham and Betti algebras}\label{sec:betti-and-derham}
In this section, we define a pair of algebras $A^\dR_G(\FI,\mass)$ and $A^\Betti_G(\FI,\mass)$
which play a central role in the remainder of this paper. 
We define these purely algebraically, but the former was shown in \cites{BLPW10,BLPW12} to control the category $\cO_G^{\dR}(\FI,\mass).$ The latter algebra is defined in a very similar way, and we will show in \Cref{ssec:algRH} that these algebras are in fact isomorphic. In \Cref{sec:geometric-categories} and \Cref{section:kirwan}, we will see that the Betti algebra actually controls the category $\cO_G^\Betti(\FI,\mass).$

\subsection{Recollections on the de Rham algebra}
Consider the quiver
\begin{equation}\label{equation:quiver-Q1}
    Q_1:= \begin{tikzcd}
 - \arrow[r, shift left,  "v"
 ] \arrow[from=r, shift left, "u"] & +.
 \end{tikzcd}
\end{equation}
Define $A^{\dR}_1 := P_{\CC}(Q_1)$ to be the path algebra of $Q_1$ over $\CC$. It is graded by path length. Let $e_\pm \in A^{\dR}_1$ be the idempotents corresponding to the vertices. 
If we set $D=\CC^\times$ to be 1-dimensional, then
the inclusion
\[
\begin{tikzcd}[row sep = 0cm]
\CC[\fd^L]\simeq \CC[d] \arrow[r] & Z(A^{\dR}_1) \\
d \arrow[r, mapsto] & uv+vu
\end{tikzcd}
\]
gives a $\fd$-action on the category $\Mod_{A^{\dR}_1}$.


More generally, we define $A^{\dR}_n := (A^{\dR}_1)^{\otimes n}.$ This algebra is also graded by path length. There is a an isomorphism $\CC[\fd^L] \cong \CC[d_1, \ldots, d_n]$ which gives rise to a $\fd$-action on $\Mod_{A^{\dR}_n}$. The idempotents in $A^{\dR}_n$ are labeled by $\alpha = i_1\dots i_n \in 2^{[n]}$: we write $e_{\alpha} := e_{i_1} \otimes \dots \otimes e_{i_n} \in A^{\dR}_n$. 
Using the bijection of $2^{[n]}$ with chambers in the coordinate hyperplane arrangement on $\RR^n,$ we can think elements of $A_n^{\dR}$ as corresponding to paths between chambers in this arrangement. This perspective, which is very useful for understanding the algebra, is taken in \cite{BLPW10}.
Between any two idempotents $\alpha$ and $\beta$ there is a minimal path 
\[
p^{dR}(\alpha, \beta) = \prod^n_{i=1} \begin{cases} 
1 & \alpha_i = \beta_i, \\ 
u_i & \text{$\alpha_i=+$ and $\beta_i=-$}, \\
v_i & \text{$\alpha_i=-$ and $\beta_i=+$}.
\end{cases}
\]

Now $(G, t, m)$ be a category $\cO$ datum, determining an inclusion $G \to D$ which we use to impose $G$-invariants. The papers \cite{BLPW10, BLPW12} study the rings
\begin{align}
A^{\dR}_G &= A^{\dR}_n \otimes_{\CC[\fg^L]} \CC, \label{eq:adrg}\\
A^{\dR}_G(\FI,-) &= e_{\cF} A^{\dR}_G( -, -) e_{\cF}, \\
A^{\dR}_G(\FI, \mass) &= A^{\dR}_G(t,-)/(A^{\dR}_G(\FI,-) e_{\cF \cap \cB^c} A^{\dR}_G(\FI,-)).
\end{align}
where in \Cref{eq:adrg}, $\CC$ is the natural augmentation module for $\CC[\fg^L].$
Note that $\Mod_{A^{\dR}_G} \simeq (\Mod_{A^{\dR}_n})^{\fg}$ has a residual $\mathfrak{f}$-action which is is inherited by $\Mod_{A^{\dR}_G(\FI,-)}$ and $\Mod_{A^{\dR}_G(\FI, \mass)}$, corresponding to the obvious $\CC[\mathfrak{f}^L] \cong \CC[\mathfrak{d}^L] \otimes_{\CC[\fg^L]} \CC$-algebra structures.

The next proposition summarizes results from \cite{BLPW10, BLPW12} that will be used later.
\begin{proposition}[{\cite{BLPW10}*{Theorem B}, \cite{BLPW12}*{Lemma 8.25}}]
\label{proposition:koszul-facts}
\text{ }
\begin{itemize}
\item $A^{\dR}_G(\FI,-)$, and $A^{\dR}_G(\FI,\mass)$ are Koszul. Moreover, $A^{\dR}_G(\FI,\mass)^! = e_{\cB \cap \cF} A^{\dR}_G(\FI,-)^! e_{\cB \cap \cF}$. 
\item $A^{\dR}_G$ and $A^{\dR}_G(\FI,-)$ are free finitely-generated $\CC[\mathfrak{f}^L]$-modules. 
\item $A^{\dR}_G(\FI,-)$ and $A^{\dR}_G(\FI,\mass)$ have finite homological dimension. 
\item $A^{\dR}_G(\FI,\mass)$ is finite-dimensional and quasi-hereditary.
\end{itemize}
\end{proposition}



Now recall the quasi-hereditary structure on $A^{\dR}_G(\FI,\mass)$ from \cite[Section 5.5]{BLPW10}. The simple modules and their projective covers are in bijection with $\cB \cap \cF \subseteq 2^{[n]}$, i.e., for each $\alpha \in \cB \cap \cF$ one has a surjection
\begin{equation} \label{eq:projcover}
P^{\dR}_{\alpha} 
= e_{\alpha} A^{\dR}_G(\FI,\mass)  \to  
A^{\dR}_G(\FI,\mass)/\langle e_{\beta} \,|\, \beta \neq \alpha \rangle = S^{\dR}_{\alpha}.
\end{equation}
Using the partial order defined in \Cref{defn:partial-order}, one can define an intermediate standard module 
\begin{equation} \label{eq:standard}
V^{\dR}_{\alpha} = P^{\dR}_{\alpha} / (P^{\dR}_{\alpha})^{> \alpha},
\end{equation}
where $ (P^{\dR}_{\alpha})^{> \alpha}$ is the submodule generated by all paths $a$ that pass through $\beta$ with $\beta > \alpha$. Let $b_{\alpha} = \{ i \in \{1, \ldots, n\} \,|\, p_{\alpha} \in H_i \}$. It is easy to see that
\begin{equation} \label{eq:standardideal}
(P^{\dR}_{\alpha})^{> \alpha} = \{ p^{\dR}(\alpha, \alpha(i)) \,|\, i \in b_{\alpha} \} \cdot A^{\dR}_G(\FI,\mass)
\end{equation}
where $\alpha(i)$ differs from $\alpha$ only in position $i$. Define $\cB_{\alpha} = \{\beta \,|\, \text{$\alpha_i = \beta_i$ for all $i \in b_{\alpha}$} \}$. Note that $\cB_{\alpha} \subseteq \cB$ and for any $\beta \in \cB_{\alpha} \cap \cF$ we have $\beta \leq \alpha$.

\begin{proposition}[\protect{\cite[Lemma 5.21, Proposition 5.23]{BLPW10}}] \text{}
\begin{itemize}
    \item The standard module $V^{\dR}_{\alpha}$ has a filtration whose associated graded is $\oplus_{\beta \in \cB_{\alpha} \cap \cF} S^{\dR}_{\beta}$.
    \item The projective module $P^{\dR}_{\alpha}$ has a filtration whose associated graded is $\oplus_{\alpha \in \cB_{\beta}\cap \cF} V^{\dR}_{\beta}$.
\end{itemize}  
\end{proposition}

\begin{corollary}\label{cor:dRmon}
There is a natural equivalence
$
\Mod_{A^{\dR}_G(\FI,\mass)} \simeq \Mod_{A^{\dR}_G(\FI,\mass)}^{\mathfrak{f}\sf{-mon}}.
$  
\end{corollary}

\begin{proof}
The simple modules are clearly $\mathfrak{f}$-monodromic. But since $A^{\dR}_G(\FI,\mass) \cong \oplus_{\alpha} P^{\dR}_{\alpha}$, it is a finite extension of simple modules and hence must also be $\mathfrak{f}$-monodromic.
\end{proof}

\begin{proposition}
Let $i^{\dR}: A^{\dR}_G(\FI,-) \to A^{\dR}_G(\FI,\mass)$. There is a recollement
\[
\begin{tikzcd}
\Mod_{A^{\dR}_G(\FI,\mass)} & 
\Mod_{A^{\dR}_G(\FI,-)}
\arrow[bend right]{l}{}
\arrow[from=l, hook, "i^{\dR}_*"]
\arrow[bend left]{l}{} & 
\Mod_{e_{\cF \cap \cB^c} A^{\dR}_G(\FI,-) e_{\cF \cap \cB^c}}
\arrow[bend right, hook', swap]{l}{}
\arrow[from=l, ""]
\arrow[bend left, hook']{l}{}
.
\end{tikzcd}
\]
In particular, $i^{\dR}_*$ is fully faithful. 
\end{proposition}
\begin{proof}
    Combine \Cref{proposition:koszul-facts} and \Cref{proposition:koszul-recollement}. Note that the proof of \Cref{cor:dRmon} shows that $A^{dR}_G(\FI,\mass)$ is in ${\sf T}_{A^{dR}_G(\FI,\mass)}(\oplus_{\alpha \in \cB \cap \cF} S^{\dR}_{\alpha})$.
\end{proof}

\begin{corollary} \label{cor:dRff}
Consider the factorization
\[
\begin{tikzcd}
 \Mod_{A^{\dR}_G(\FI,-)}^{\mathfrak{f}\sf{-mon}} \arrow[r, hook] & \Mod_{A^{\dR}_G(\FI,-)} \\
 \Mod_{A^{\dR}_G(\FI,\mass)}^{\mathfrak{f}\sf{-mon}} \arrow[r , "\sim"] \arrow[u, swap, "(i^{\dR}_*)^{\frf{\sf -mon}}"] & \Mod_{A^{\dR}_G(\FI,\mass)} \arrow[u, swap, "i^{\dR}_*"]
\end{tikzcd}
\]
coming from \Cref{cor:monfact}. The map $(i^{\dR}_*)^{\frf{\sf -mon}}$ is fully faithful. 
\end{corollary}

\subsection{The Betti algebra}
Consider the quiver
\begin{equation}
    Q'_1:= \begin{tikzcd}
 - \arrow[r, shift left,  "\ell"
 ] \arrow[from=r, shift left, "r"] & +.
 \end{tikzcd}
\end{equation}
There is an inclusion
\begin{equation}\label{eq:betti-action-on-algebra}
\begin{tikzcd}[row sep = 0cm]
\CC[m] \arrow[r] & Z(P_{\CC}(Q'_1)) \\
m \arrow[r, mapsto] & 1-(\ell r + r \ell)
\end{tikzcd}
\end{equation}
Define $A^{\Bet}_1 = P_{\CC}(Q_1')[m^{-1}]$ and let $e_\pm \in A^{\Bet}_1$ be the idempotents corresponding to the vertices. If we take $D=\CC^\times$ to be 1-dimensional, then identifying $\CC[D^L]\cong\CC[m^{\pm 1}]$ gives a Betti $D$-action on $\Mod_{A^{\Bet}_1}$.

More generally, define $A^{\Bet}_n = (A^{\Bet}_1)^{\otimes n}$. This is naturally an algebra over $\CC[D^L] \cong \CC[m_1^{\pm 1}, \ldots, m_n^{\pm 1}]$. Note that for each $\alpha = i_1\dots i_n \in 2^{[n]}$, there is an idempotent $e_\alpha:= e_{i_1} \otimes \dots \otimes e_{i_n} \in A^{\Bet}_n$.  Between any two idempotents $\alpha$ and $\beta$ there is a minimal path 
\[
p^{\Bet}(\alpha, \beta) = \prod^n_{i=1} \begin{cases} 
1 & \alpha_i = \beta_i, \\ 
l_i & \text{$\alpha_i=+$ and $\beta_i=-$}, \\
r_i & \text{$\alpha_i=-$ and $\beta_i=+$}.
\end{cases}
\]

Let $(G, t, m)$ be a category $\cO$ datum. 
Define the rings
\begin{align}
A^{\Bet}_G &= A^{\Bet}_n \otimes_{\CC[G^L]} \CC, \label{eq:abetg}\\
A^{\Bet}_G(t,-) &= e_{\cF} A^{\Bet}_G(-, -) e_{\cF}, \\
A^{\Bet}_G(\FI, \mass) &= A^{\Bet}_G(\FI,-)/(A^{\Bet}_G(\FI,-) e_{\cF \cap \cB^c} A^{\Bet}_G(\FI,-)).
\end{align}
where in \Cref{eq:abetg}, $\CC$ is the $\CC[G^L]$-module corresponding to the inclusion $\{1\}\hookrightarrow G^L.$
Note that $\Mod_{A^{\Bet}_G} \simeq (\Mod_{A^{\Bet}_n})^{G_{\Bet}}$ has a residual Betti $F$-action, inherited by $\Mod_{A^{\Bet}_G(\FI,-)}$ and $\Mod_{A^{\Bet}_G(\FI, \mass)}$, corresponding to the obvious $\CC[F^L] \cong \CC[D^L] \otimes_{\CC[G^L]} \CC$-algebra structures.

The following proposition is the key technical result in this section.
\begin{proposition}\label{prop:bettimon} There is a natural equivalence
$\Mod_{A^{\Bet}_G(\FI,\mass)} \simeq \Mod_{A^{\Bet}_G(\FI,\mass)}^{F_{\Bet}\sf{-mon}}$
\end{proposition}

Before proving \Cref{prop:bettimon}, we note that it can be combined with the algebraic Riemann--Hilbert correspondence of \Cref{prop:monodromic-Riemann-Hilbert} to prove the following theorem.
\begin{theorem}\label{thm:Betti-recollement}
Let $i^{\Bet}: A^{\Bet}_G(\FI,-) \to A^{\Bet}_G(\FI,\mass)$. There is a recollement
\[
\begin{tikzcd}
\Mod_{A^{\Bet}_G(\FI,\mass)} & 
\Mod_{A^{\Bet}_G(\FI,-)}
\arrow[bend right]{l}{}
\arrow[from=l, hook, "i^{\Bet}_*"]
\arrow[bend left]{l}{} & 
\Mod_{e_{\cF \cap \cB^c} A^{\Bet}_G(\FI,-) e_{\cF \cap \cB^c}}
\arrow[bend right, hook', swap]{l}{}
\arrow[from=l, ""]
\arrow[bend left, hook']{l}{}
.
\end{tikzcd}
\]
\end{theorem}

\begin{proof}
By \Cref{proposition:ff-recollement}, it suffices to show that $i^{\Bet}_*$ is fully faithful. However, this follows from \Cref{prop:monodromic-Riemann-Hilbert} and \Cref{cor:dRff} since we have a factorization
\[
\begin{tikzcd}
\Mod_{A^{\dR}_G(\FI,-)}^{\mathfrak{f}\mon} \arrow[r, "\RH","\sim"'] & \Mod_{A^{\Bet}_G(\FI,-)}^{F_{\Bet}\mon} \arrow[r, hook] & \Mod_{A^{\Bet}_G(\FI,-)} \\
\Mod_{A^{\dR}_G(\FI,\mass)}^{\mathfrak{f}\mon} \arrow[u, swap, "(i^{\dR}_*)^{\frf{\sf -mon}}"]  \arrow[r, "\RH","\sim"'] & \Mod_{A^{\Bet}_G(\FI,\mass)}^{F_{\Bet}\mon} \arrow[r , "\sim"'] \arrow[u, swap, "(i^{\Bet}_*)^{\frf{\sf -mon}}"] & \Mod_{A^{\Bet}_G(\FI,\mass)}. \arrow[u, swap, "i^{\Bet}_*"]
\end{tikzcd}
\]
\end{proof}

The rest of this section is devoted to the proof of \Cref{prop:bettimon}. As in \Cref{cor:dRmon}, it would suffice to show that $A^{\Bet}_G(\FI,\mass)$ is a finite extension of simple modules.

\begin{remark}
In \Cref{prop:monodromic-Riemann-Hilbert}, we will establish a Riemann--Hilbert isomorphism $A^{\Bet}_G(\FI,\mass) \cong A^{\dR}_G(\FI,\mass),$ from which we will deduce that the former is hence a finite dimensional quasi-hereditary algebra. Unfortunately, to prove \Cref{prop:monodromic-Riemann-Hilbert}, one needs certain power series in $(1-m_i)$ to converge. This is the content of \Cref{prop:bettimon}. 
\end{remark}

Simple modules $S^{\Bet}_{\alpha}$, projective modules  $P^{\Bet}_{\alpha}$, and standard modules $V^{\Bet}_{\alpha}$ can be defined by replacing $\dR$ with $\Bet$ in equations \Cref{eq:projcover,,eq:standard,,eq:standardideal}. To verify that they satisfy the expected properties, we need to first establish a few facts about $A^{\Bet}_G(\FI,\mass)$. The following are the Betti analogues of \cite[Proposition 3.8, Proposition 3.9, Corollary 3.10]{BLPW10}.        
\begin{proposition}
The algebra $A^{\Bet}_G$ is a free finitely-generated $\CC[F^L]$-module, i.e.,
\[
A^{\Bet}_G = \bigoplus_{\alpha, \beta \in 2^{[n]}} e_{\alpha} A^{\Bet}_G e_{\beta} = \bigoplus_{\alpha, \beta \in 2^{[n]}} \CC[F^L] \cdot p^{\Bet}(\alpha,\beta)
\]
Moreover, one has
\[
p^{\Bet}(\alpha, \beta) p^{\Bet}(\beta, \gamma) = \prod^n_{i=1} (1-m_i)^{\theta_i(\alpha, \beta, \gamma)} \cdot p^{\Bet}(\alpha, \gamma)
\]
where
\[
\theta_i(\alpha, \beta, \gamma) = \begin{cases} 1 & \alpha_i = \gamma_i \neq \beta_i, \\
0. & \text{else}
\end{cases}
\]
\end{proposition}

\begin{corollary}
Suppose that 
\[
a = \prod^n_{i=1} (1-m_i)^{n_i} \cdot p^{\Bet}(\alpha, \gamma)
\]
and $n_i \geq \theta_i(\alpha, \beta, \gamma)$ for all $1 \leq 1 \leq n$. Then $a$ can be represented by a $\CC$-linear combination of paths that pass through $\beta$. In particular, if $\beta \in \cB^c \cap \cF$ then the image of $a$ in $A^{\Bet}_G(\FI,\mass)$ is $0$.
\end{corollary}

The following is the Betti analogue \cite[Lemma 5.21]{BLPW10} and is proved in the same way.

\begin{lemma}\label{lemma:v-alpha}
The module $V^{\Bet}_{\alpha}$ has a basis consisting of all $p^{\Bet}(\alpha, \beta)$ with $\beta \in \cF \cap \cB_{\alpha}$. In particular, $V^{\Bet}_{\alpha}$ has a filtration whose associated graded is $\oplus_{\beta \in \cB_{\alpha} \cap \cF} S^{\Bet}_{\beta}$.
\end{lemma}

The following is the first part of \cite[Proposition 5.23]{BLPW10} and is proved in the same way. It is sufficient to complete the proof of \Cref{prop:bettimon}.
\begin{lemma}\label{lemma:kernel-v-alpha}
The projective module $P^{\Bet}_{\alpha}$ has filtration such that the associated graded is a quotient of $\oplus_{\alpha \in \cB_{\beta}\cap \cF} V^{\Bet}_{\beta}$.
\end{lemma}

\begin{proof}
For each $\gamma$ define $(P^{\Bet}_{\alpha})^{\gamma}$  to be the submodule of $P^{\Bet}_{\alpha}$ generated by all paths passing through $\gamma$. Moreover we can define 
\begin{align*}
(P^{\Bet}_{\alpha})^{\geq \beta} &= \sum_{\gamma \geq \beta} (P^{\Bet}_{\alpha})^{\gamma} & (P^{\Bet}_{\alpha})^{> \beta} &= \sum_{\gamma > \beta} (P^{\Bet}_{\alpha})^{\gamma}.
\end{align*}
Note that $(P^{\Bet}_{\alpha})^{\geq \alpha} = P^{\Bet}_{\alpha}$ and $V^{\Bet}_{\alpha} = P^{\Bet}_{\alpha}/ (P^{\Bet}_{\alpha})^{> \alpha}$. These give a decreasing filtration $(P^{\Bet}_{\alpha})^{\geq \bullet}$ indexed by the partially ordered set $(\cB \cap \cF)_{\geq \alpha}$. The associated graded pieces are the quotients
\[
M^{\beta}_{\alpha}= (P^{\Bet}_{\alpha})^{\geq \beta} / (P^{\Bet}_{\alpha})^{> \beta}
\]
for $\beta \in (\cB \cap \cF)_{\geq \alpha}$. When $\alpha \not\in \cB_{\beta} \cap \cF$ we have that $M^{\beta}_{\alpha} = 0$. Otherwise there is a surjective map
\[
\begin{tikzcd}[row sep = 0cm]
P^{\Bet}_{\beta} \arrow[r] & M^{\beta}_{\alpha}  \\
a \arrow[r,mapsto] & p^{\Bet}(\alpha, \beta) \cdot a
\end{tikzcd}
\]
with kernel containing $(P^{\Bet}_{\beta})^{> \beta}$. In particular we have a surjection $V^{\Bet}_{\beta} \to M^{\beta}_{\alpha}$.
\end{proof}

\begin{remark}
The dimension estimate used to finish the proof of \cite[Proposition 5.23]{BLPW10} follows from identifying $A^{\dR}_G(\FI,\mass)$ with a convolution algebra built from the homology of toric varieties. It is possible to get the same dimension estimate for $A^{\Bet}_G(\FI,\mass)$ by identifying it with a convolution algebra defined using the rational K-theory of toric varieties, as is implicitly done in \cite{GH-2O}. From this perspective the algebraic Riemann--Hilbert correspondence discussed below is just the Chern character.
\end{remark} 

\subsection{An algebraic Riemann--Hilbert correspondence} \label{ssec:algRH}
The most primitive form of the algebraic Riemann--Hilbert correspondence is the following proposition.
\begin{proposition}\label{prop:RH-primitive}
There is an equivalence
\begin{equation}\label{eq:RH}
RH: \Mod_{A^{\Bet}_n}^{D_{\Bet}\mon} \simeq \Mod_{A^{\dR}_n}^{\fd \mon}
\end{equation}
compatible with the residual actions of
$
(\IndCoh((D^L)^\wedge_1), \otimes) \simeq (\IndCoh((\fd^L)^\wedge_0), \otimes).$
\end{proposition}

\begin{proof}
By \Cref{cor:monfact}, we have equivalences
\begin{align*}
\Mod_{A^{\Bet}_n}^{D_{\Bet}\mon} &\simeq
    \varinjlim_k \IndCoh(N_k(1, D^L)) \otimes_{\IndCoh(D^L)} \Mod_{A^{\Bet}_n} \simeq \varinjlim_k \Mod_{A^{\Bet}_n/\cI^{k+1}}, \\
   \Mod_{A^{\dR}_n}^{\fd \mon} &\simeq \varinjlim_k \IndCoh(N_k(0, \fd_n^L)) \otimes_{\IndCoh(\fd^L)} \Mod_{A^{dR}_n} \simeq \varinjlim_k \Mod_{A^{\dR}_n/\cJ^{k+1}}.
\end{align*}
where $\cI \subset A^{\Bet}_n$ (resp., $\cJ \subset A^{\dR}_n$) is the two-sided ideal generated by $(1-m_i)$ (resp., $d_i$) for $1 \leq i \leq n$.

Thus we have reduced constructing $RH$ to the observation that in \cite[II.3.2]{Malgrange-diffeqs} Malgrange constructed isomorphisms 
\begin{equation}\label{eq:RH}
\begin{tikzcd}[row sep = 0cm]
A^{\Bet}_n/\cI^{k+1} \arrow[r, "rh_k"] & A^{\dR}_n/\cJ^{k+1} \\
\ell_i \arrow[r, mapsto] & u_i \\
r_i \arrow[r, mapsto] & \frac{1-\exp(2\pi i d_i)}{d_i}v_i
\end{tikzcd}.
\end{equation}
Compatibility with the residual actions follows from the fact that $m_i \mapsto \exp(2\pi i d_i)$.
\end{proof}

By taking $G$-invariants of \Cref{eq:RH}, we obtain an equivalence
\begin{equation}\label{eq:RH-Ginvt}
(\Mod_{A^{\Bet}_n}^{D_{\Bet}\mon})^{G_{\Bet}} \simeq 
(\Mod_{A^{\dR}_n}^{\fd \mon})^{\fg}.
\end{equation}

\begin{corollary}
There is an equivalence of categories
\begin{equation}\label{eq:RH-monodromic}
\Mod_{A^{\Bet}_G}^{F_{\Bet}\mon} \simeq 
\Mod_{A^{\dR}_G}^{\mathfrak{f}\mon}.
\end{equation}
\end{corollary}

\begin{proof}
Consider the pullback squares
\[
\begin{tikzcd}
(F^L)^\wedge_1 \arrow[r] \arrow[d] &  F^L \arrow[r] \arrow[d] &  1 \arrow[d] & 
(\mathfrak{f}^L)^\wedge_0 \arrow[r] \arrow[d] &  \mathfrak{f}^L \arrow[r] \arrow[d] &  0 \arrow[d] \\
(D^L)^\wedge_1 \arrow[r] & D^L_n \arrow[r] & G^L, & (\fd^L_n)^\wedge_0 \arrow[r] & \fd^L \arrow[r] & \fg^L.
\end{tikzcd}
\]
The left-hand side of \Cref{eq:RH-monodromic} may be obtained by pulling back from $G^L$ to $(F^L)_1^\wedge.$ On the other hand, this pullback may be accomplished in two steps, first pulling back to $(D^L)^\wedge_1$ and then to $(F^L)_1^\wedge,$ giving an equivalence between the left-hand sides of \Cref{eq:RH-monodromic} and \Cref{eq:RH-Ginvt}. An equivalence between the right-hand categories may be produced in the same way by considering the right-hand pullback diagram.
\end{proof}

By examining equation \Cref{eq:RH} we can easily see that the Riemann--Hilbert correspondence is compatible with idempotents.
\begin{corollary}\label{prop:monodromic-Riemann-Hilbert}
There is a commutative diagram
\[
\begin{tikzcd}
\Mod_{A^{\dR}_G(\FI,-)}^{\mathfrak{f}\mon} \arrow[r, "\RH", "\sim"'] & \Mod_{A^{\Bet}_G(\FI,-)}^{F_{\Bet}\mon} \\
\Mod_{A^{\dR}_G(\FI,\mass)}^{\mathfrak{f}\mon} \arrow[u, swap, "({i}^{\dR}_*)^{\frf{\sf -mon}}"]  \arrow[r, "\RH", "\sim"'] & \Mod_{A^{\Bet}_G(\FI,\mass)}^{F_{\Bet}\mon}. \arrow[u, swap, "({i}^{\Bet}_*)^{F_\Bet{\sf -mon}}"]
\end{tikzcd}
\]
Moreover, by combining this with \Cref{cor:dRmon} and \Cref{prop:bettimon}, we produce an equivalence
\[
\Mod_{A^{\dR}_G(\FI,\mass)} \simeq \Mod_{A^{\Bet}_G(\FI,\mass)}.
\]
\end{corollary}

\section{The geometric categories}\label{sec:geometric-categories}
We now begin to relate the algebraic constructions of previous sections to geometry. In this section, we recall results about the category of sheaves on $\CC^n,$ stratified by its toric orbits, together with its toric-equivariant structure.
\subsection{Categories of sheaves}
\label{section:sheaves}
Let $M$ be a manifold. For consistency with parts of the literature (e.g. \cites{Kashiwara-Schapira, GPS3}, we will always assume when discussing sheaves that $M$ is real-analytic. Let $\sh(M) \in \bCat_\infty$ be the (very large) category of sheaves on $M$ valued in the (large) category  $\Mod_\mathbb{C}$. Given $\mathcal{F} \in \sh(M)$, its \emph{singular support} (or \emph{microsupport}) $SS(F) \subset T^*M$ is defined as in \cite[Sec.\ 5.1]{Kashiwara-Schapira}. For $\Lambda \subset T^*M$, let $\sh_\Lambda(M) \subset \sh(M)$ be the full subcategory on objects whose microsupport is contained in $\Lambda$.
%

Given a manifold $M$, one can define a presheaf of stable categories on $T^*M$ by the prescription
\begin{equation}
    U \mapsto \mush_{T^*M}^{pre}(U):= \sh(M)/ \{ F \in \sh(M) \mid SS(F) \cap U = \emptyset \}.
\end{equation}
The sheafification of this presheaf shall be denoted by $\mush_{T^*M}(-)$, and is called the sheaf (of categories) of \emph{microlocal sheaves} (or {\em microsheaves} for short) on $T^*M$. 

Given $\Lambda \subset T^*M$ conic, there is a sheaf of full subcategories $\mush_\Lambda(-) \subset \mush_{T^*M}(-)$ which is defined by the prescription 
\begin{equation}
    \mush_\Lambda(U \cap \Lambda):= \{ \mathcal{F} \in \mush_{T^*M}(U) \mid SS(\mathcal{F}) \subset \Lambda \cap U \}.
\end{equation}
Under the assumption (which will always be satisfied in this paper) that $\Lambda$ is nice stratified isotropic, $\mush_\Lambda(-)$ takes values in the full subcategory $\PrLRomega \subset \bCat_\infty$ (objects compactly generated, morphisms have left and right adjoints). 

Given an exact symplectic manifold $(X,\lambda)$ equipped with a stable polarization $\pol$ and a conic (for the Liouville vector field) Lagrangian $\Lambda\subset X,$ \cite{Shende-weinstein,nadler2020sheaf} use the theory of microlocal sheaves to define a sheaf of categories $\mush_{\Lambda,\pol}$ on $X.$ It has the property \cite[Remark 9.25]{nadler2020sheaf} that if $X=T^*M$ is a cotangent bundle and $\pol^{fib}$ is the cotangent fiber polarization, then there is an equivalence $\mush_{\Lambda}$ and $\mush_{\Lambda,\tau}$ between the two sheaves (on $T^*M$) of categories just discussed. We will henceforth drop the polarization $\tau$ from our notation where it is understood.

A priori, the definition of $\mush_{\Lambda}$ on a Lagrangian $\Lambda$ polarized symplectic manifold is sensitive to the choice of primitive $\lambda$ for the symplectic form on $X$.
In this paper, we will be interested in studying the Lagrangian $\LL_n\subset T^*\CC^n,$ which is conic for two different Liouville vector fields: the canonical Liouville structure on $T^*\CC^n$, which dilates the cotangent fibers; and the conical Liouville structure on $\CC^{2n}$, which dilates both base and fiber coordinates equally. The sheaf of categories $\mush_{\LL_n}$ will not depend on the choice of Liouville field, thanks to the following fact. In the following lemma and its corollary, we will write $\mush_{\Lambda,\lambda}$ for the sheaf of categories defined using Liouville form $\lambda$ (where the fixed choice of cotangent fiber polarization is assumed).
\begin{lemma}\label{lem:liouville-independence}
Let $\lambda_s:= (1-s)p_idq^i- s q_idp^i= \lambda_0- d(s p_iq^i)$ by Liouville forms on $T^*\mathbb{R}^n$ and let $\LL:= 0_{\mathbb{R}^n} \cup T^*_0 \subset T^*\mathbb{R}^n$. Then the sheaf of categories $\mush_{\LL,\lambda_s}$ is independent of $s$.
\end{lemma}
\begin{proof}
Recall \cite[p.\ 2]{nadler2020sheaf} that if $(X, \lambda)$ is Liouville and $L \subset X$ is conic, then $\mush_L$ is defined by viewing $\{0\} \times L$ as a Legendrian in the contact manifold $(\mathbb{R} \times X, dt + \lambda)$. 

For any contactomorphism $(V, \xi) \to (U, \eta)$ and Legendrian $L \subset V$, we have $\mush_{L, \xi}(-)\simeq \mush_{\phi(L), \eta}$, where it is understood the Maslov data on the source is pulled back from the target. But the contactomorphism $(\mathbb{R} \times T^*\mathbb{R}^n, dt+ \lambda_s) \to (\mathbb{R} \times T^*\mathbb{R}^n, dt+ \lambda_0)$ taking $(t, q,p) \mapsto(t- sp_iq^i, q, p)$ fixes $\{0\} \times \LL$ pointwise, giving the desired claim.
\end{proof}
\begin{corollary}
    Consider the family (in $s$) of Liouville forms $\lambda^\CC_s = (1-s)wdz-szdw$ on the holomorphic symplectic manifold $T^*\CC.$ Then $\mush_{\LL_1,\Re(\lambda^\CC_s)}$ is independent of $s.$
\end{corollary}
\begin{proof}
 Identify $T^*\mathbb{C} \simeq T^*\mathbb{R}^2$ as usual by the map $(z,w)\mapsto (re(z), im(z), re(w),-im(w))$. Let $\lambda^{\mathbb{C}}_s= (1-s) wdz-s zdw$. Then $re(\lambda^{\mathbb{C}}_s)$ is identified with $\lambda_s$, and $s$-independence of the category follows from \Cref{lem:liouville-independence}.
\end{proof}

If $M$ is a complex manifold  and $\Lambda \subset T^*M$ is a conic complex Lagrangian, then $\mush_\Lambda(-)$ 
pushes forwards to a sheaf of categories on the topology of $\mathbb{C}^\times$-invariant conic complex open subsets. As explained in \cite{perverse-microsheaves}, for a holomorphic polarization $\pol,$ the resulting sheaf of categories $\mush_{\Lambda,\tau}(-)$ 
is equipped with a t-structures. The heart $\mush_\Lambda^{\heartsuit}(-)$
is a sheaf of abelian categories; its objects are the {\em perverse microsheaves}.


\subsection{Group actions}\label{sec:stacks}
%

We will also want to consider Betti equivariance for microsheaf categories. Let $G$ be a complex algebraic group acting on a complex manifold $M$ and $\Lambda\subset T^*M$ a $G$-invariant conic Lagrangian. Then the convolution action of $\Loc(G)$ on $\Sh_\Lambda(M)$ makes $\Sh_\Lambda(M)$ into a Betti G-category. 
\begin{definition}
    We write $\Sh_\Lambda^G(M):=\Sh(M)^{G_\Betti}$ for the Betti $G$-invariants of the Betti $G$-category $\Sh_\Lambda(M).$
\end{definition}
\begin{remark}
    Without imposing the Lagrangian singular support condition, the category $\Sh(M)$ more naturally carries an action of the monoidal category $\Sh(G)$ of all sheaves on $G$, and the invariants for this action give a definition of sheaves on the stack $G/M.$ This is the setting of the calculations in \cite{NS-higgs}, to which we shall refer below. However, after imposing the microsupport condition $\Lambda$, we ensure that this $\Sh(G)$-action factors through $\Loc(G),$ so that we do not lose anything by studying the Betti action in this case.
\end{remark}

Moreover, if $U\subset T^*M$ is conic and $G$-invariant, then the action of $\Loc(G)$ will preserve the condition that singular support is contained outside of $U$, thus inducing a Betti $G$-action on the category $\mush^{pre}_{\Lambda}(U).$
\begin{definition}
    Let $\Lambda\subset T^*M$ be a $G$-invariant conic Lagrangian. We write $\mush^G_\Lambda$ for the sheaf of categories on $\Lambda/G$ obtained by sheafifying $U\mapsto \mush^{pre}_{\Lambda}(U)^{G_\Betti}$ on $G$-invariant open subsets of $\Lambda.$
\end{definition}
We could equivalently have defined $\mush^G_\Lambda(U) = \mush_\Lambda(U)^{G_\Betti}$; see \cite{NS-higgs}*{Lemma 4.2}.

Now suppose that $\tilU\subset \mu^{-1}(0)$ is an open subset on which the $G$-action is free, such that $U:=\tilU/G$ is a Liouville submanifold of the stack $\mu^{-1}(0)/G=T^*(M/G).$ As a cotangent bundle, $T^*M$ comes equipped with a canonical polarization by cotangent fibers.
\begin{lemma}[\cite{NS-higgs}*{Corollary 4.5}]\label{lem:descend-polarization}
    The cotangent fiber polarization on $T^*M$ descends to a polarization on the manifold $U.$
\end{lemma}
Let $\pol_G$ denote the polarization on $U$ defined in \Cref{lem:descend-polarization}. We may learn from \cite{NS-higgs} a recipe for computing the microsheaf category on $U$ in terms of microsheaves in $T^*M$.
\begin{proposition}\label{prop:musheaves-from-stack}
    Let $\tilU$ as above, and let $\widetilde{\Lambda}\subset \tilU\subset T^*M$ be a conic Lagrangian. Then $\mush_{\widetilde{\Lambda}}(\widetilde{\Lambda})$ has a Betti $G$-action, and there is an equivalence
    \[
    \mush_{\widetilde{\Lambda}}^G(\widetilde{\Lambda})\simeq \mush_{\Lambda,\pol_G}(\Lambda).
    \]
\end{proposition}
\begin{proof}
    \cite{NS-higgs}*{Corollary 4.9} describes a convolution action of $\Sh(G)$ on $\mush_{T^*M},$ and \cite{NS-higgs}*{Lemma 4.11} shows that $\mush_\Lambda(\Lambda)$ is equivalent to the $\Sh(G)$-invariant category $\mush_{\tilLambda}(\tilLambda)^{\Sh(G)}.$ However, since the $\Sh(G)$-action on $\mush_{\tilLambda}(\tilLambda)$ factors through the (oplax monoidal) projection map $\Sh(G)\to \Loc(G)$ 
    which is left adjoint to the inclusion $\Loc(G)\to \Sh(G),$ we deduce that the $\Sh(G)$-action factors through a Betti $G$-action, and the $\Sh(G)$-invariants are the Betti $G$-invariants.
\end{proof}

\subsection{Perverse (micro)sheaves in $T^*\mathbb{C}$} 
\label{sec:sheafcalc}

We now begin to study the geometry of the Lagrangians 
\[\LL_1 := \CC\cup T^*_0\CC\subset T^*\CC, \qquad \LL_n:=(\LL_1)^n\subset (T^*\CC)^n=T^*\CC^n.\]
The following fundamental calculation was first observed by Deligne, with proofs by Verdier \cite{Verdier}, Galligo-Granger-Maisonobe \cite{GGM}, and MacPherson-Vilonen \cite{MV-elementary}.
\begin{theorem}\label{thm:perverse-description-1d}
    There is an equivalence of categories
\begin{equation}\label{eq:betti-is-sheaves-basic}
\Sh_{\LL_1}(\CC)\simeq \Mod_{A^\Betti_1},
\end{equation}
exchanging the perverse t-structure on the left-hand side with the usual t-structure on the right-hand side, and exchanging the Betti $\CC^\times$-action induced on the left-hand side by the action $\CC^\times\curvearrowright\CC$ with the Betti $\CC^\times$-action on the right-hand side described by \Cref{eq:betti-action-on-algebra}.
\end{theorem}
\begin{proof}
    The equivalence of the abelian hearts of these categories is a direct corollary of \cite{MV-elementary}*{Theorem 3.3} (described explicitly in the example following the proof there), except that the theorem there is stated with the assumption of finite-dimensionality. However, the assumption of finite-dimensional microstalks is never used in the proof, which continues to hold in the infinite-dimensional case.

    In other words, \cite{MV-elementary}*{Theorem 3.3} proves that $\Perv_{\LL_1}(\CC)\simeq \Mod_{A_1^\Betti}^\heart.$ It remains to show that $\Sh_{\LL_1}(\CC)$ is in fact the derived category of its perverse heart, for which it is sufficient to check that the objects $P_\Psi,P_\Phi$ of $\Sh_{\LL_1}$ corepresenting the ``generic stalk'' and ``generic microstalk over 0'' (i.e., nearby and vanishing cycles) functors are in fact perverse. 
    This is the case, since $P_\Psi\simeq j_!(\exp_*{\ul \CC}),$ 
    where $\exp:\CC\to \CC^\times$ is the universal covering map and $j:\CC^\times\to \CC$ is the inclusion, and $P_\Phi$ can be obtained from $P_\Psi$ by Fourier transform.

    Finally, the statement about Betti $\CC^\times$-actions follows from the identification of the monodromy on $\End(P_\Psi)\sim \CC[m^\pm]$ as multiplication by $m.$
\end{proof}

We also want to (micro)localize the above equivalence. Let $\LL_+ = \CC \setminus \{0\}, \LL_- = T^*_0\CC\setminus \{0\}\subset \LL_1$ be the open subsets obtained by deleting one or the other component of $\LL_1.$
\begin{corollary}
    The following vertical maps are an equivalence of diagrams of categories, respecting Betti $\CC^\times$-actions and t-structures:
    \[
    \begin{tikzcd}
        \mush_{\LL_1}(\LL_-)\isoarrow{d}&
        \ar[l]\ar[r] \mush_{\LL_1}(\LL_1)\isoarrow{d} &
        \mush_{\LL_1}(\LL_+)\isoarrow{d}\\
        \Mod_{e_-A_1^\Betti e_-} & \ar[r]\ar[l] \Mod_{A_1^\Betti} & \Mod_{e_+A_1^\Betti e_+}
    \end{tikzcd}
    \]
\end{corollary}

By taking tensor products, we obtain a version of the above in $n$ dimensions, which we may state as follows. 
\begin{definition}\label{def:pm1-poset}
    Equip $\{1,+,-\}$ with the poset structure $-\gets 1\to +,$ and write $\bP:=\{1,+,-\}^n$ for the product poset.
    For $\alpha\in \{1, +, -\}^n,$ let $\LL_\alpha := \prod_{i=1}^n\LL_{\alpha_i}\subset (T^*\CC)^n\simeq T^*\CC^n$ be the product Lagrangian.
    \end{definition}
Observe that $\LL_{1^n} = \LL_n,$ and that there is an open embedding $\LL_{\alpha'}\subset \LL_\alpha$ whenever $\alpha'$ is obtained from $\alpha$ by changing some $1$'s to $\pm$'s. As $\alpha$ varies, we thus obtain a $\bP$-diagram of categories $\alpha\mapsto\mush_{\LL_n}(\LL_\alpha),$ with the evident restriction maps.
Similarly, let $e_1=1, e_+,e_-$ be the idempotents in $A_1^\Betti,$ and for $\alpha\in \{1,+,-\}^n,$ let $e_\alpha = \otimes_i e_{\alpha_i}\in (A_1^\Betti)^{\otimes n}\simeq A_n^\Betti.$ Then we get another $\bP$-diagram of categories $\alpha\mapsto \Mod_{e_\alpha A_n^\Betti e_\alpha}.$
 %
\begin{corollary}\label{cor:musheaves-ndimensions}
    There are equivalences of $\bP$-shaped diagrams of Betti $(\CC^\times)^n$-categories with t-structures, given by equivalences
    \begin{equation}\label{eq:betti-is-sheaves-n}
    \mush_{\LL_n}(\LL_\alpha)\simeq \Mod_{e_\alpha A_n^\Betti e_\alpha}.
    \end{equation}
    commuting with the restriction maps.
\end{corollary}


%
%
\begin{corollary}\label{cor:eqvt-musheaves-ndimensions}
    Let $G\subset (\CC^\times)^n$ be a subtorus. Then 
    there are equivalences of $\bP$-shaped diagrams of categories with t-structures
    $\mush_{\LL_n}^G(\LL_\alpha)\simeq \Mod_{e_\alpha A_n^\Betti(G,-,-) e_\alpha}.$
\end{corollary}
\begin{proof}
    This follows from \Cref{cor:musheaves-ndimensions} by taking $G$-invariants.
\end{proof}

\subsection{The de Rham category} \label{ssec:RH}
We will now describe a portion of the Riemann--Hilbert correspondence associated to the complex manifold $\CC$ with singular-support condition $\LL_1\subset T^*\CC.$
Consider the category $\Dmod^{reg}_{\LL_1}(\CC)$ of regular D-modules on $\CC$ with singular support in $\LL_1.$
\begin{theorem}[\cite{Malgrange-diffeqs}*{Th\'eor\`eme (2.1)}]\label{thm:derhamcat-basic}
    The compact objects in the category $\Dmod^{reg}_{\LL_1}(\CC)$
    are equivalent to the category of finite-dimensional $A_1^\dR$-modules on which all eigenvalues of $vu$ have real part in the interval $[0,1).$
\end{theorem}

%
The respective categories of perverse sheaves and D-modules on $\CC,$ even after imposing the singular support condition $\LL_1,$ are not equivalent: a perfect module over $A_1^\Betti$ may not be finite-dimensional (as indeed the regular representation of the infinite-dimensional $\CC$-algebra $A_1^\Betti$ is not). The Riemann--Hilbert correspondence, described in \cite{Malgrange-diffeqs}*{Th\'eor\`eme (3.1)}, gives an equivalence between the category described in \Cref{thm:derhamcat-basic} and the category $\Modfd_{A_1^\Betti}$ of finite-dimensional $A_1^\Betti$-modules.



The de Rham category can be described more simply if we pass to monodromic modules. In our language, one of the main calculations of \cite{MvB} is the following:
\begin{theorem}[\cite{MvB}*{Theorem 6.3}]\label{thm:MvB}
There is an equivalence $\Mod_{A^{\dR}_n}^{\fd\sf{-mon}} \simeq \Dmod(\CC^n)^{D_{\dR} \sf{-mon}}$. 
\end{theorem}
Using \Cref{thm:MvB} and \Cref{thm:derhamcat-basic}, we can now describe the geometric origin of the Riemann--Hilbert isomorphism of \Cref{prop:RH-primitive}.

\begin{corollary}
    There is an equivalence of categories
    \[
    \Sh_\LL(\CC^n)^{D_\Betti{\sf -mon}}\simeq\Dmod(\CC^n)^{D_\dR\sf{-mon}},
    \]
    given by the map \Cref{eq:RH}.
\end{corollary}

\section{Betti category $\cO$}
\label{section:kirwan}

In this section, we describe the category $\mush_{\Oskel}(\Oskel)$ of microlocal sheaves on the category $\cO$ skeleton $\Oskel$ of the hypertoric variety $\htvar_G(\FI).$ We will begin by describing the passage from the Lagrangian 
$\LL=\LL_n$
to the extended core
$\LL_G(\FI,-)$ of the hypertoric variety $\fX_G(\FI)$;
afterward, we will pass from $\LL_G(\FI,-)$ to the category $\cO$ skeleton $\Oskel.$

For the reader's convenience, we recall the following notation from \Cref{sec:hyperplane}.
\begin{notation}\text{ }
    \begin{itemize}
        \item  $\LL_G=\LL_G(-,-)$ denotes the stacky Lagrangian $\LL_n/G\subset T^*(\CC^n/G).$
        \item $\Lstab=(\LL_n\cap\mu^{-1}(\FI,0,0))/G\subset \htvar_G(\FI)$ is the extended core of $\htvar_G(\FI)$; equivalently, $\Lstab=\htvar_G(\FI)\cap \LL_G\subset \LL_G$ is the $\FI$-stable locus in $\LL_G.$
        \item The category $\cO$ skeleton $\Oskel\subset \LL_G(\FI,-)$ is the subset which remains bounded under the action of $\CC^\times_m.$
    \end{itemize}
\end{notation}

\subsection{Stable microrestriction}

We now begin to study the category of microsheaves on the Lagrangian $\LL_G(\FI,-)\subset \htvar_G(\FI).$ Our main tool will be the embedding of $\htvar_G(\FI)$ in the stack $T^*(\CC^n/G),$
to which we will apply the theory of \Cref{sec:stacks}.
\begin{definition}\label{def:Oskel-polarization}
    We will write $\htpol$ for the polarization on $\htvar_G(\FI)$ induced from the embedding $\htvar_G(\FI)\subset T^*(\CC^n/G)$ by \Cref{lem:descend-polarization}.
\end{definition}

Our first goal is to describe the category $\mush_{\LL_G(\FI,-),\htpol}(\LL_G(\FI,-)),$ which we will accomplish by taking a limit over microsheaf categories on an open cover of the Lagrangian $\LL_G(\FI,-).$ 


Recall from \Cref{prop:moment-polytopes} that there is a bijection between irreducible components of $\Stab$ and chambers of the hyperplane arrangement $\cH(\FI),$ realized by a map $\mu^\RR:\Stab\to \hpspace$ whose restriction to each irreducible component of $\Stab$ is its toric moment map to the corresponding polytope in $\cH(\FI).$ The space $\hpspace,$ stratified by the hyperplane arrangement $\cH(\FI),$ has an obvious open cover indexed by strata, where to each stratum $S$ we associate the union of strata containing $S$ in their closure.


\begin{definition}
    Let $\faces$ denote the stratum poset of $\hpspace,$ stratified by $\cH(\FI).$
\end{definition}
The poset $\faces$ indexes the aforementioned open cover of $\hpspace,$ and hence also an open cover of $\Stab$, which we now describe.
Observe that there is a canonical embedding of $\faces$ as a subposet of the poset $\bP$ defined in \Cref{def:pm1-poset}, using the natural identification of $\bP$ as the stratum poset for the coordinate hyperplane arrangement on $\RR^n$. Thus, to each $\alpha\in\faces,$ we can associate an open subset $\LL_\alpha\subset \LL,$ as described in \Cref{def:pm1-poset}. Write $\LL_{G,\alpha}:=(\mu^{-1}(\FI,0,0)\cap \LL_\alpha)/G\subset \htvar_G(\FI)$ for the induced Lagrangian subset of $\htvar_G(\FI).$

\begin{lemma}
    The $\faces$-shaped diagram of open subsets $\LL_{G,\alpha}\subset\Stab$ and their inclusions is an open cover of the Lagrangian $\Stab.$
\end{lemma}
\begin{proof}
    The Lagrangian $\LL_{G,\alpha}$ is precisely the preimage in $\Stab$ of the open neighborhood of the stratum $\alpha$ in the hyperplane arrangement $\cH(\FI).$
\end{proof}

As $\mush_{\Stab,\htpol}$ is a sheaf of categories on $\Stab,$ the above open cover determines a limit description of its global sections:
\begin{corollary}
    There is an equivalence
    \begin{equation}\label{eq:limit-description}
    \mush_{\Stab,\htpol}(\Stab)\simeq \varinjlim_{\faces}\mush_{\Stab,\htpol}(\LL_{G,\alpha})
    \end{equation}
\end{corollary}

By the results of the previous section, we can reduce the above to a purely algebraic computation.
\begin{proposition}
    There are equivalences of categories
    \begin{equation}\label{eq:diagrams-final}
    \mush_{\Stab,\htpol}(\LL_{G,\alpha})\simeq \Mod_{e_\alpha A_G^\Betti e_\alpha},
    \end{equation}
    commuting with restriction maps and thus determining an equivalences of $\faces$-shaped diagrams of categories.
\end{proposition}
\begin{proof}
    Let $\LL(\FI,-)=\bigcup_{\alpha\in \faces}\LL_\alpha.$ Then \Cref{prop:musheaves-from-stack} determines an equivalence 
    \[\mush_{\LL(\FI,-)}^G\simeq \mush_{\Stab,\htpol}\] 
    of sheaves of categories on $\Stab$ (where implicitly on the left-hand side we have used the identification of $G$-invariant open subsets of $\LL(\FI,-)$ with open subsets of $\LL_G(\FI,-)$). The identification of $\faces$-shaped diagrams \Cref{eq:diagrams-final} now follows from \Cref{cor:eqvt-musheaves-ndimensions}.
\end{proof}

The following lemma, which is joint with McBreen and Webster, asserts that the assignment $F\mapsto \Mod_{e_F A_G^\Betti e_F}$ satisfies descent. 
\begin{lemma}[{\cite{gammage2019homological}*{Lemma C.2}}]
    There is an equivalence
    \begin{equation}\label{eq:4.22}
        \varprojlim_{\faces}\Mod_{e_F A_G^\Betti e_F} \simeq \Mod_{A^{\Betti}_G(\FI, -)}.
    \end{equation}
\end{lemma}

\begin{remark}
    The statement of \cite[Lemma C.2]{gammage2019homological} refers to an algebra $\mathscr{Q}$ associated to a hyperplane arrangement $\cH.$ When $\cH=\cH(\FI)$ is the hyperplane arrangement defined in \Cref{defn:H(t)}, the definition of this algebra in \cite{gammage2019homological} coincides with $A^{\Betti}(G, \FI, -)$. Similarly, $\mathscr{Q}^F= e_F A^{\Betti}_G(\FI, -)e_F$. 
\end{remark}

\begin{corollary}\label{cor:betti-allfeasible-presentation}
    There is an equivalence of categories
    \begin{equation}\label{eq:betti-allfeasible}
    \mush_{\Lstab,\htpol}(\Lstab)\simeq \Mod_{A_G^\Betti(\FI,-)}.
    \end{equation}
\end{corollary}
\begin{proof}
The left-hand side admits a limit description given by \Cref{eq:limit-description}, which is computed by the equivalences \Cref{eq:diagrams-final} and \Cref{eq:4.22}.
\end{proof}
\begin{remark}
    By tracing through the above equivalences, one can see that the equivalence \Cref{eq:betti-allfeasible} is induced by an equivalence between the algebra $A_G^\Betti(\FI,-)$ and the endomorphism algebra of the objects corepresenting microstalk functors at smooth strata of the Lagrangian $\Stab.$
\end{remark}

We note the following corollary of our results, which ought to go by the name of ``categorical Kirwan surjectivity'' (a version of which was proved in the de Rham setting in \cite{BPW}*{Theorem 5.31}); it witnesses the essential surjectivity of restriction from stacky microsheaves to the $\FI$-stable locus.
\begin{corollary}
    The restriction functor $\mush^G_{\LL}(\LL)\to \mush_{\Stab,\htpol}(\Stab)$ has a fully faithful left adjoint.
\end{corollary}
\begin{proof}
    By \Cref{cor:eqvt-musheaves-ndimensions} and \Cref{cor:betti-allfeasible-presentation}, this restriction functor is equivalent to the functor $\Mod_{A_G^\Betti}\to \Mod_{A_G^\Betti(\FI,-)},$ which does indeed admit a fully faithful left adjoint, given by tensoring with the $A_G^\Betti-A_G^\Betti(\FI,-)$-bimodule $A_G^\Betti e_{\cF}.$
\end{proof}

\subsection{Betti category $\mathcal{O}$}
We are now ready for the computation of the Betti category $\cO.$
\begin{theorem}\label{thm:betti-category-O}
    There is an equivalence of categories
    \begin{equation}\label{eq:betti-catO}
        \mush_{\Lstab,\htpol}(\Oskel)\simeq \Mod_{A_G^\Betti(t,m)},
    \end{equation}
    intertwining the microlocal perverse t-structure with the standard t-structure.
\end{theorem}
\begin{proof}
    By stop removal, the left-hand side of \Cref{eq:betti-catO} can be computed as a quotient
    \[
    \mush_{\Lstab,\htpol}(\Oskel)\simeq \mush_{\Lstab,\htpol}(\Stab)/\cB^c,
    \]
    where $\cB^c\subset \mush_{\Lstab,\htpol}(\Stab)$ is the full subcategory generated by corepresentatives of the microstalk functors at the smooth points of $\mass$-unbounded components of $\Stab.$ Under the equivalence \Cref{eq:betti-allfeasible}, $\cB^c$ corresponds to the full subcategory of $\Mod_{A_G^\Betti(\FI,-)}$ 
    split-generated by the module $e_{\cF} A_G^\Betti e_{\cF\cap \cB^c},$
    where we write $e_{\cF\cap \cB^c}$ for the idempotent corresponding to those chambers which are both $\FI$-feasible and $\mass$-unbounded.

Using the equivalence $\mush_{\Stab,\htpol}(\Stab)\simeq \Mod_{A_G^\Betti(\FI,-)},$ we thus have a recollement
\begin{equation}\label{eq:geometric-recollement}
\begin{tikzcd}
\mush_{\Lstab,\htpol}(\Oskel) & 
\Mod_{A_G^\Betti(\FI,-)}
\arrow[bend right]{l}{}
\arrow[from=l, hook]
\arrow[bend left]{l}{} & 
\Mod_{e_{\cF \cap \cB^c} A_G^{\Betti}(t,-) e_{\cF \cap \cB^c}}
\arrow[bend right, hook', swap]{l}{}
\arrow[from=l, ""]
\arrow[bend left, hook']{l}{}
.
\end{tikzcd}
\end{equation}
We conclude from \Cref{thm:Betti-recollement} that the left-hand category in \Cref{eq:geometric-recollement} is equivalent to $\Mod_{A_G^\Betti(\FI,\mass)}.$
\end{proof}


\begin{proof}[Proof of \Cref{theorem:main-theorem}]
    By \Cref{thm:betti-category-O} we have $\mush_{\Lstab,\htpol}(\Oskel)\simeq \Mod_{A_G^\Bet(t,m)}$. By \Cref{prop:monodromic-Riemann-Hilbert} $\Mod_{A_G^\Bet(t,m)}= \Mod_{A_G^{\dR}(t,m)}$. 
\end{proof}

\begin{remark}
    In light of \Cref{thm:betti-category-O}, we can now reinterpret \Cref{thm:Betti-recollement} as the statement that the stop removal functor $\mush_{\LL_G(\FI,-),\pol_G}(\LL_G(\FI,-))\to \mush_{\LL_G(\FI,-),\pol_G}(\LL_G(\FI,\mass))$ becomes fully faithful when restricted to the subcategory generated by simple objects, i.e., components of $\LL_G(\FI,-).$ This is obvious for the compact components of $\LL_G(\FI,-)$ but nontrivial in general. (See the following remark for a counterexample.)
\end{remark}

\begin{remark}\label{rem:derived-recollement}
\Cref{thm:Betti-recollement} states that if we take $e:=e_{\cF\cap \cB^c}$ to be the idemponent in $A:=A_G^\Betti(\FI,-)$ corresponding to feasible unbounded sign vectors, then the ideal $AeA\subset A$ is stratifying in the sense of \Cref{defn:ff-recollement}.
If instead we had taken $e$ to be the idempotent corresponding to all noncompact components of $\LL(\FI,-),$ this would no longer be true, and to obtain a recollement as in \Cref{eq:geometric-recollement} it would be necessary to replace $A/AeA$ with the derived quotient defined in \Cref{defn:derived-quotient}.
In the case where $G=\CC^\times\xrightarrow{\Delta}(\CC^\times)^{n+1}=D$ is the diagonal copy of $\CC^\times$ inside $D,$ this construction gives an explicit presentation of the dg algebra $C_*(\Omega \PP^n)$ governing the category $\Loc(\PP^n)$ of local systems on projective space.
\end{remark}

\section{Fukaya categories}\label{sec:fukaya}

The fundamental theorem of Ganatra--Pardon--Shende relates partially wrapped Fukaya categories to microlocal sheaf categories. As a consequence:
\begin{theorem}\label{theorem:gps-equivalence}
    Given a datum $(G, \FI, \mass)$ for category $\cO$, there is an equivalence of categories 
    \begin{equation}
        \cW(\htvar_G(\FI), \partial_\infty \LL_G(\FI,\mass)) \simeq (\mush_{\LL_G(\FI,\mass)}(\LL_G(\FI,\mass))^c)^{op}
    \end{equation}
    compatible with the microlocal perverse t-structure.
\end{theorem}
\begin{proof}
    $\htvar_G(\FI)$ is $I$-holomorphic and the relative skeleton $\Oskel$ is an analytic isotropic subvariety. $\htvar_G(\FI)$ is Weinstein by \Cref{lem:HK-Weinstein}. The claim follows by \cite{GPS3}*{Theorem 1.4}. 
\end{proof}

Combining \Cref{theorem:main-theorem} and \Cref{theorem:gps-equivalence}, we deduce \Cref{corollary:main-fukaya-intro}. In particular, there is a t-exact equivalence of categories
\begin{equation}\label{equation:fuk-cat-o}
     \cW(\htvar_G(\FI), \partial_\infty \LL_G(\FI,\mass))^{op}\simeq D^b(\cO_G^{\dR}(\FI, \mass)).
\end{equation}

\begin{remark}
    For any holomorphic Weinstein manifold $X$ and conic holomorphic (possibly singular) Lagrangian $\LL \subset X$, there is always a map 
\begin{equation}\label{equation:derived-equiv-heart}
    D^b((\cW(X, \partial_\infty \LL)^c)^\heartsuit) \to \cW(X, \partial_\infty \LL)^c
\end{equation}
but it may fail to be an equivalence. For example, if $X= T^*\mathbb{P}^1$ and $\Lambda$ is the zero section, then this map is the inclusion of $\Mod_\mathbb{C}$ into $\Loc(S^2)$. This example may be understood as illustrating the failure of the heart of the t-structure to capture the derived behavior of the quotient algebra described in \Cref{rem:derived-recollement}.
\end{remark}
\begin{remark}\label{remark:self-opposite}
    The $\SS$-action of weight-$2$ described in \Cref{subsection:liouville-geometry} furnishes an anti-symplectic involution of $\htvar_G(\FI)$ which negates the Liouville form and fixes $\LL_G(\FI,\mass)$ set-wise. It follows that $\cW(\htvar_G(\FI), \partial_\infty \LL_G(\FI,\mass))$ and $\cW(\htvar_G(\FI))$ are self opposite.
\end{remark}

%

\subsection{Formality}\label{subsec:formality} Recall that an $A_\infty$ algebra is said to be \emph{formal} if it is equivalent to its cohomology algebra. The paradigmatic example of formality is a classical result of Deligne--Griffiths--Mumford--Sullivan \cite{deligne1975real}, who proved that the $A_\infty$ algebra of chains on a compact Kähler manifold $M$ is formal in characteristic zero. This algebra is equivalent to the algebra of Floer cochains on the holomorphic Lagrangian submanifold $M\subset T^*M,$ which serves as some motivation for the following
well-known ``folk-conjecture'' in symplectic topology. 
\begin{conjecture*}[Formality conjecture]\label{conjecture:formality}
Let $M=(M, g, I, J, K)$ be a hyperkähler manifold. Let $L_1, \dots, L_k$ be a collection of closed $I$-complex submanifolds which are Lagrangian with respect to $\omega_{\mathbb{C}}:= \omega_J+ i\omega_K$. Then the Floer--Fukaya $A_\infty$ algebra $CF^*(\oplus_i L_i,\oplus_i L_i)$, defined with respect to $\omega_J$ for the real symplectic manifold $(M, \omega_J)$ is formal in characteristic zero.
\end{conjecture*}

In our situation, we prove a slightly more general result, which also incorporates the non-compact Lagrangian components of the category $\cO$ skeleton.
\begin{corollary}
\label{corollary:formality}
 Fix a datum $(G, \FI, \mass)$ for category $\cO$. Let $S$ be the direct sum of the irreducible components of $\LL_G(\FI,\mass)$. Then $CF^*(S,S)$ is formal in characteristic zero. 
\end{corollary}
\begin{proof}
    Note that $CF^*(S,S)\simeq \operatorname{End}_{(\cW(\htvar_G(\FI), \partial_\infty \LL_G(\FI,\mass))}(S)$. Hence it is equivalent to check formality for the endomorphism algebra of the image of $S$ under \eqref{equation:fuk-cat-o}. In other words, we must verify that the direct sum of the simples in $\mathcal{O}_G^{\dR}(\FI, \mass)$ has formal Yoneda-dg algebra, which is a known fact (e.g.\    $\mathcal{O}_G^{\dR}(\FI, \mass)$ is Koszul \cite{BLPW12}*{Corollary 4.10}; this implies the desired claim e.g.\ by \cite[Rmk.\ 4.6]{BLPW16}.)
\end{proof}

It may be useful to contrast \Cref{corollary:formality} with previous results on the hyperkähler formality conjecture. The first work going beyond \cite{deligne1975real} is probably Seidel--Thomas \cite{seidel2001braid} which established the conjecture for skeleta of ALE spaces of type $A_n$. This was subsequently extended to type $D_n$ by Etg\"u--Lekili \cite{etgu2017koszul}.  These proofs are in some sense algebraic, since they establish that the relevant $A_\infty$ algebras are \emph{intrinsically formal}, meaning that any $A_\infty$ structure with the given cohomology ring is formal. 

An entirely different method for proving formality was introduced by Abouzaid--Smith \cite{abouzaid2016symplectic} and also used by Mak--Smith \cite{mak2021fukaya}. They construct a non-commutative degree-$1$ Hochschild cocycle (understood as a non-commutative vector field) satisfying a purity condition.  By an argument of Seidel, the existence of such a vector field implies formality of the Floer algebra of the Lagrangians, due to the presence of an extra grading induced by the cocycle.
%
This is the approach to the hypertoric formality conjecture taken in \cite{LLM}.

On the other hand, our \Cref{corollary:formality} follows from \Cref{theorem:main-theorem}, since the image of $S$ under \eqref{equation:gps-equivalence-intro} and \eqref{equation:main-theorem-equivalence} is known to have formal endomorphism algebra. However, from our perspective, the origin of this formality is Hodge theory: the endomorphism algebra of $S$ is equivalent to the endomorphism algebra of a $D$-module which admits a lift to mixed Hodge modules. Purity of these mixed Hodge modules furnishes an extra grading which precludes the existence of any higher $A_\infty$ operations. We refer to \cite{Web-gencat-O}*{\S 2.5} for a discussion of this philosophy in de Rham category $\cO.$

\subsection{Koszul duality}
\label{ssec:Koszul}
As in the previous section, let $S$ be the direct sum of the irreducible components of $\Oskel,$ as an object of $\cW(\htvar_G(\FI),\partial_\infty\Oskel),$ and let $P$ be the direct sum of their projective covers (i.e., the direct sum of cocores to the components of $S$). Then by applying the main theorem of \cite{BLPW10}, we deduce a relationship between their endomorphism algebras.
\begin{theorem}
    The algebras $\End(S)$ and $\End(P)$ are Koszul bidual.
\end{theorem}
\begin{proof}
    This is the statement of Koszulity for the algebra $A^\Betti_G(\FI,\mass)$, which we may deduce from the same statement about $A^\dR_G(\FI,\mass)$ in \cite{BLPW10}*{Theorem B}.
\end{proof}
However, we also prove a new result. 
Recall from \Cref{lem:skeleton-is-compacts} that the skeleton of the Liouville manifold $\htvar_G(\FI)$ is equal to the union of {\em compact} components of $\Oskel.$
Let $\Core=\bigoplus_\alpha S_\alpha$ be the direct sum of those components,
as an object of the (fully) wrapped Fukaya category $\cW(\htvar_G(\FI)).$
Similarly, let $\Cocore=\bigoplus_\alpha P_\alpha$ be the direct sum of cocores to those components.
\begin{theorem}\label{thm:koszul}
    The algebras
        $\End_{\cW(\htvar_G(\FI))}(\Cocore)$ and $\End_{\cW(\htvar_G(\FI))}(\Core)$
        are Koszul bidual.
\end{theorem}
\begin{proof}
    Let $A$ denote the algebra $A^\Betti_G(\FI,\mass),$ which is the endomorphism algebra of the indecomposable projective objects in $\cW(\htvar_G(\FI,\mass), \partial_\infty\Oskel)^{op},$ and write $A^!$ for its Koszul dual, which, as we have seen above, is the endomorphism algebra of the simple components of $\Oskel.$
    
    Let $\cN\subset \cF\cap \cB$ be the set of bounded feasible sign vectors corresponding to noncompact components of $\Oskel,$ and write $e_\cN\in A$ for the corresponding idempotent (and $e_{\cN^c}$ for the complementary idempotent). By stop removal, the first algebra mentioned in the theorem is $A/^L(Ae_\cN A).$ The second algebra is $e_{\cN^c}A^! e_{\cN^c}.$ The Koszul duality between these is the statement of \Cref{cor:kdid}.
\end{proof}

The phenomenon of Koszul duality in symplectic geometry has been studied earlier in e.g.\ \cites{etgu2017koszul,li2019exact,li2019koszul}; guided by those investigations, we highlight here several known Floer-theoretic consequences of \Cref{thm:koszul}. 

First, we deduce that the symplectic cohomology of the Weinstein manifold $\htvar_G(\FI)$ can be calculated purely in terms of compact components of the Lagrangian skeleton.

\begin{definition}
    We will write $B_G(\FI):=\End_{\cW(\htvar_G(\FI))}(\Core)$ for the endomorphism algebra of compact components of the skeleton of $\htvar_G(\FI)$.
\end{definition}
The components $S_\alpha,S_{\alpha'}$ intersect cleanly in the manifold $S_\alpha\cap S_\alpha'$, and the morphisms among these Lagrangians in the Fukaya category may be computed as cohomologies
\begin{equation}\label{eq:compact-morphisms}
\Hom_{\cW(\htvar_G(\FI))}(S_\alpha,S_{\alpha'})\simeq H^*(S_\alpha\cap S_{\alpha'}).
\end{equation}
The spaces \Cref{eq:compact-morphisms} are the matrix coefficient blocks of the matrix algebra $B_G(\FI),$ and the compositions among them (i.e., the multiplications in the algebra $B_G(\FI)$) are given by cap product in the triple intersection followed by Gysin pushforward, as in \cite{BLPW10}*{\S 4.3}. (A priori, it might have seemed necessary to work with some model of cochains, rather than cohomology, on the intersections $S_\alpha\cap S_{\alpha'},$ and to keep careful track of higher $A_\infty$ operations, but by \Cref{corollary:formality} the algebra $B_G(\FI)$ is actually formal.)

Using the algebra $B_G(\FI),$ we give a new computation of $SH^*(\htvar_G(\FI)).$

\begin{corollary}
    There is an equivalence
    \begin{equation}\label{eq:sh-equivalence}
        SH^{*+\dim_\CC(\htvar_G(\FI))}(\htvar_G(\FI))\simeq HH_{-*}(B_G(\FI))^\vee.
    \end{equation}
    between the symplectic cohomology of $\htvar_G(\FI)$ and the dual of the Hochschild homology of the algebra $B_G(\FI).$
\end{corollary}
\begin{proof}
    By \cite{ganatra2013symplectic}*{Theorem 1.1}, the open-closed map gives an isomorphism between the left-hand side of \Cref{eq:sh-equivalence} and the Hochschild homology $HH_*(\cW(\htvar_G(\FI)))$ of the wrapped Fukaya category $\cW(\htvar_G(\FI)),$ which is the category of perfect modules over $\End_{\cW(\htvar_G(\FI))}(\Cocore).$ (Here we have implicitly appealed to \Cref{remark:self-opposite}). From the Koszul duality equivalence of \Cref{thm:koszul}, this is equivalent to the thick subcategory of $\Mod_{B}$ generated by the simple $B$-modules. However, since $B$ is finite-dimensional, this latter category is dual to $\Perf_B$ by \cite{Campbell-Koszul}*{Proposition 4.9}. We conclude that the Hochschild homology of $\End_{\cW(\htvar_G(\FI))}(\Cocore)$ is dual to the Hochschild homology of $B_G(\FI),$ as in \cite{Campbell-Koszul}*{Theorem 4.16}.
\end{proof}
\begin{example}
    If $\htvar_G(\FI) = T^*\PP^1$ (see \Cref{example:classical-sl2}) then we have equivalences 
    \[
    HH_{-*}(B_G(\FI))^\vee \simeq HH_{-*}(H^*(\mathbb{P}^1))^\vee \simeq  H^{-*}(L\mathbb{P}^1)^\vee.
    \]
    Meanwhile, $SH^{*+\dim(\htvar_G(\FI))}(\htvar_G(\FI)) \simeq H_{-*}(L\mathbb{P}^1)$. 
\end{example}

\begin{remark}
    In fact, one can show that $SH^*(\htvar_G(\FI))$ is always locally finite, so \eqref{eq:sh-equivalence} remains true without dualizing the right hand side.  Indeed, we saw in the proof of \Cref{thm:koszul} that $\cW(\htvar_G(\FI)) \simeq \Perf_{A/^L Ae_\mathcal{N}A}$ for $A=A^\Bet_G(\FI, \mass)$. We know from \cite{GPS2} that $A/^L Ae_\mathcal{N}A$ is smooth, i.e.\ the diagonal bimodule $\Delta_{A/^L Ae_\mathcal{N}A}$ is perfect.  Since $A/^L Ae_\mathcal{N}A$ and hence $(A/^L Ae_\mathcal{N}A) \otimes (A/^L Ae_\mathcal{N}A)^{\op}$ is locally finite (\Cref{cor:kdid}), any perfect $A/^L Ae_\mathcal{N}A$-bimodule has locally finite endomorphism ring. But  $SH^*(\htvar_G(\FI))= HH^*(\cW(\htvar_G(\FI)))= \End(\Delta_{A/^L Ae_\mathcal{N}A})$ by \cite{ganatra2013symplectic, GPS2}.
\end{remark}

Koszul duality also implies the following generation result. 

\begin{corollary}\label{cor:split-gen-by-compact}
    Any compact Lagrangian brane in $\htvar_G(\FI)$ is split-generated by components of the compact core.
\end{corollary}
Since a compact Lagrangian defines a proper module over $\cW(\htvar_G(\FI)) = A/^L(A e A)$, \Cref{cor:split-gen-by-compact} follows immediately from the following lemma: 

\begin{lemma}
    In the notation of \Cref{prop:derivedKDID}, $\cT_{\RMod_{A/^L(A e A)}}(T)$ is the category of proper modules over $A/^L(A e A)$.
\end{lemma}
\begin{proof}
    Note that since $A$ is a finite-dimensional quasi-hereditary algebra and has finite homological dimension we have that $\cT_{\RMod_A}(S)= \Perf_A = \Prop_A$. Moreover the fully faithful embedding ${i_*: \RMod_{A^L/(AeA)} \hookrightarrow \RMod_{A}}$ sends proper modules to proper modules and $i_*(T) = S e^c$. Thus it suffices to show that $$\Im i_* \cap \Prop_A = \Im i_* \cap \cT_{\RMod_A}(S) = \cT_{\RMod_A}(Se^c).$$ This follows from the fact that the left inverse of $i_*$ is $i^*$ and $i^*(Se)= 0$.
\end{proof}

Another application of Koszul duality is the existence of dilations in symplectic cohomology. Let $n=\dim_\CC(\htvar_G(\FI)),$ and 
recall that the main results of \cites{ganatra2013symplectic, ST-CY} furnish the algebras $A/^L(A e A)=\End_{\cW(\htvar_G(\FI))}(\Cocore)$ and  $B_G(\FI)=\End_{\cW(\htvar_G(\FI))}(\Core)$ with smooth (respectively, proper) $n$-Calabi--Yau structures,
which determine BV-algebras structure on their Hochschild cohomologies by \cite[Theorem 3.4.3]{ginzburg2006calabi} and \cite{tradler2008batalin}. (As explained in \cite{Brav-Rozenblyum}, the identification of Hochschild cohomology, which is an $\EE_2$-algebra, with Hochschild homology, which carries a circle action, given by the CY structure results in a framed $\EE_2$-algebra, i.e., a homotopy BV algebra.) 
Moreover, the Koszul duality between these algebras induces a BV-algebra isomorphism of their Hochschild cohomologies by \cites{Han-hochschild}.
Finally, since $\htvar_G(\FI)$ is Weinstein, 
we know from \cite[\S 5]{seidel2008biased} that
the closed-open map $$SH^*(\htvar_G(\FI)) \to HH^*(A/^L(A e A), A/^L(A e A))$$ is also a BV-algebra isomorphism.
We summarize the above discussion as follows:
\begin{lemma}\label{lem:BV-iso}
    There is a BV-algebra isomorphism $SH^*(\htvar_G(\FI))\simeq HH^*(B_G(\FI)).$
\end{lemma}
Following \cite{seidel2012symplectic}, if $B$ is a BV-algebra (with BV-operator $\Delta$) and $\alpha$ is a 1-cocycle in $B$ satisfying the equation $\Delta(\alpha)=1,$ then we call $\alpha$ a {\em dilation} on $B$. We can now state our second application of Koszul duality.
\begin{corollary}\label{cor:dilation}
    $SH^*(\htvar_G(\FI))$ admits a dilation.
\end{corollary}

\begin{proof}
    Using the equivalence of \Cref{lem:BV-iso}, we need to produce a dilation in the Hochschild cohomology of $B_G(\FI).$ Observe that $B_G(\FI)$ is formal, it admits a canonical Hochschild 1-cocycle $\eu_{B_G(\FI)},$ which acts on a degree-$d$ element of $B_G(\FI)$ as multiplication by $d$. The result now follows from the following lemma.
\end{proof}



\begin{lemma}
    Suppose that $B$ is a formal $A_\infty$-algebra equipped with a proper $n$-Calabi--Yau structure $\langle -, - \rangle$. Then $\Delta(\frac{1}{n}\eu_B)=1$. 
\end{lemma}
\begin{proof}
     Let $x \in B$ be a n-cocycle and let $\alpha \in CC^1(B,B)$ be a Hochschild 1-cocycle in cohomological degree zero.  
     By definition of $\Delta$ \cite{tradler2008batalin}, we have $\langle \Delta(\alpha), x \rangle= \langle \alpha(x), 1 \rangle$. Meanwhile $\eu (x)= n x$, so we have $\langle \Delta(\eu_B), x \rangle = \langle n x, 1 \rangle = n \langle 1, x \rangle$. Hence $\Delta(\eu_B)= n \cdot 1$.
\end{proof}



The existence of a dilation on symplectic cohomology puts strong constraints on the symplectic geometry of $\htvar_G(\FI).$ Standard consequences include: non-existence of exact Lagrangian $K(\pi,1)$ \cite{seidel2012symplectic}, improved intersection numbers for Lagrangians \cite{seidel2012symplectic}, upper bounds on the size of a collection of of disjoinable Lagrangian spheres \cite{seidel2014disjoinable}.

\subsection{The Fukaya--Seidel category}\label{sec:heuristic}

There is a well-established heuristic for interpreting the generalized category $\cO$ of \cite{BLPW16} as the Fukaya--Seidel category of a Lefschetz fibration. 
In the case of a 2-block nilpotent-slice in type A, this was made precise in \cite{mak2021fukaya}.
%
%
In the present context, we learn from \Cref{lemma:lefschetz-fiber-skeleton} that
%
the partially wrapped Fukaya category $\cW(\htvar_G(\FI), \partial_\infty \LL_G(\FI,\mass))$ should be interpreted as a Fukaya--Seidel category 
of $\htvar_G(\FI)$ equipped with the Lefschetz fibration given by the $J$-holomorphic moment map for $\CC^\times_m.$
We explain here how certain representation-theoretic features of hypertoric category $\cO$ should be understood from this perspective.

We have seen that $\cO^{\dR}_G(\FI, \mass)$ is a highest-weight category. As explained in \Cref{subsection:highest-weight}, it comes equipped distinguished objects $\{L_\alpha\}$, $\{P_\alpha\}$, $\{V_\alpha\}$ (resp.\ the simples, projectives and standards) indexed by a poset $\cJ$, each of which forms a basis for the Grothendieck group. 
In \Cref{table:dictionary}, we recall how these bases are expected to relate to objects of study in symplectic topology.
\begin{table}[h]
\begin{tabular}{ |c|c|c| } 
 \hline
 Representation theory & Symplectic topology  \\ 
 \hline
 simple modules & irreducible components of the skeleton \\
 indecomposable projective modules & cocores / linking disks\\
standard (or Verma) modules & Lefschetz thimbles  \\ 
 \hline
\end{tabular}
\vspace{0.3cm}
\caption{A (well-known, and partly heuristic) dictionary}\label{table:dictionary}
\end{table}

The first two lines of \Cref{table:dictionary} are strictly true as written, and follow from tracing through the equivalences which enter into the proof of \Cref{corollary:main-fukaya-intro}.  
The third line is only a heuristic. First, we have not even discussed how to makes thimbles into objects of our partially wrapped model for the Fukaya--Seidel category (although see \cite{GPS2}*{\S 8.6} for one approach.) More seriously, our method only gives an algebraic description of Verma modules as iterated cones on the irreducible components of the category $\cO$ skeleton (see \Cref{lemma:v-alpha}). Geometrically, one may think of these objects as arising as an ``adiabatic limit'' of thimbles. It would be interesting to reverse this process, and prove that these algebraic objects in the Fukaya category are represented by honest Lagrangians.

\subsection{Examples}\label{ssec:examples} We end with some extended examples.

\begin{figure}[h]
\includegraphics[width=6cm]{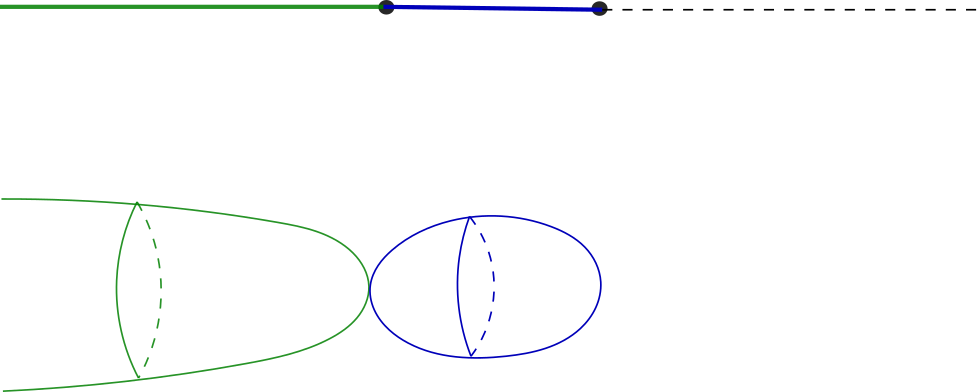}
\caption{A hyperplane arrangement and the corresponding category $\cO$ skeleton. The underlying Weinstein manifold is $T^*\mathbb{P}^1$. There are two simple objects,  namely the cotangent fiber over $\infty \in \mathbb{P}^1$ and the zero section.}\label{figure:tstarp1}
\centering
\end{figure}
\begin{example} \label{example:classical-sl2}Consider the hyperplane arrangement drawn in \Cref{figure:tstarp1}. This can be realized concretely by the exact sequence 
\[
\begin{tikzcd}
    1\ar[r]&
    G=\mathbb{C}^\times \ar[r,"\Delta"]&
    (\mathbb{C}^\times)^2 \ar[r]&
    F=(\mathbb{C}^\times)^2/\mathbb{C}^\times\ar[r]&
    1,
\end{tikzcd}
\]
where $\FI: G \to \mathbb{C}^\times$ and $\mass: F \to \mathbb{C}^*$ are nonzero characters. 
In this case, $\htvar_G(\FI) = T^*\PP^1,$ and the Hamiltonian action of $F$ on $T^*\PP^1$ is induced by the standard $\CC^\times$-action on $\PP^1.$ The Lagrangian $\Oskel$ is the conormal to the Schubert stratification $\PP^1=\{0\}\sqcup \CC.$

We have $\mush_{\Oskel}(\Oskel) \simeq \sh_{\Oskel}(\mathbb{P}^1),$ and the equivalence respects the perverse t-structure.
This is the classical category $\cO$ for $SL(2),$ which may be written as
\[
\cO_G^\Betti(\FI,\mass) \simeq \left\{\
\begin{tikzcd}
    V\ar[r, shift left, "x"]&
    W\ar[l, shift left, "y"]
\end{tikzcd} \mid xy = 0
\right\}.
\]
Let $L_1$ be the cotangent fiber (green) and let $L_2$ be the zero section (blue). Viewed as objects of $\cO^{\Bet}_G(\FI, \mass)$, they correspond to the respective representations $\CC\rightleftarrows 0$ and $0\rightleftarrows \CC.$
We have $V_{1}= L_{1}$, $V_2= P_2$. These fit into nontrivial extensions $0\to V_{1}\to V_2 \to L_2 \to 0$ and $0 \to V_2 \to P_1 \to V_1 \to 0$. 


\end{example}

\begin{figure}[h]\label{figure:path841}
\includegraphics[width=6cm]{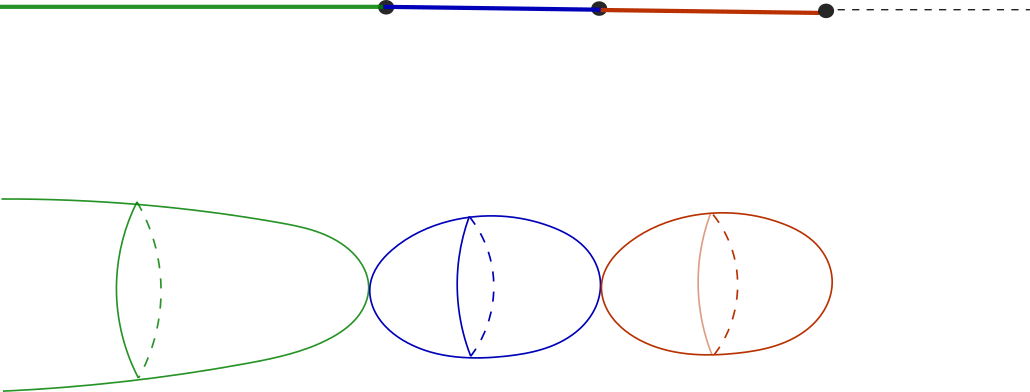}
\caption{A hyperplane arrangement and the corresponding category $\cO$ skeleton. The underlying Weinstein manifold is a plumbing of two copies of $T^*\mathbb{P}^1$. The category $\cO$ skeleton has three irreducible components (the zero zero sections and a cotangent fiber), corresponding to the three simple objects of category $\cO$.}\label{figure:plumbing}
\centering
\end{figure}

\begin{example}
    Now consider the arrangement drawn in \Cref{figure:plumbing}, for which $\htvar_G(\FI)$ is a plumbing of two copies of $T^*S^2$, and $\LL_G(\FI,\mass)$ is the union of both zero sections and a cotangent fiber. 
In this case, the category
\[
\cO_G^\Betti(\FI,\mass) \simeq \left\{\
\begin{tikzcd}
    U\ar[r, shift left, "x_1"] &
    V\ar[r, shift left, "x_2"]\ar[l, shift left, "y_1"]&
    W\ar[l, shift left, "y_1"]
\end{tikzcd} \mid x_1y_1=x_2y_2, x_2y_2 = 0
\right\}
\]
has three simples $L_1< L_2< L_3$ (green, blue, red). All nontrivial morphisms between the simples are in degree $1$. We have $V_1= L_1$ and nontrivial extensions $0 \to L_1 \to V_2 \to L_2 \to 0$ and $0 \to V_2 \to V_3 \to L_3[-1] \to 0$. Alternatively, we can write $V_2= \Cone(L_2[-1] \to L_1=V_1)$ and $V_3= \Cone(L_3[-2] \to V_2)$; this description may be more compelling to symplectic geometers, since it gives a recipe for representing the $V_i$ by honest Lagrangians by surgery. 
\end{example}

\begin{example}
    Consider now the hyperplane arrangement drawn in \Cref{figure:drawing-4d}, where we have used the natural metric on $\RR^2$ to identify the codirection $\mass$ with the direction indicated by the arrow. 
    As a Weinstein manifold, the space $\htvar_G(\FI)$ may be obtained by plumbing $T^*\PP^2$ with the cotangent bundle to a Hirzebruch surface along a shared copy of $\PP^1.$
    With respect to the ordering induced by $\mass$, the five simples are $L_1$ (red), $L_2$ (green), $L_3$ (blue), $L_4$ (purple), and $L_5$ (orange). 

\begin{figure}[h]
\includegraphics[width=6cm]{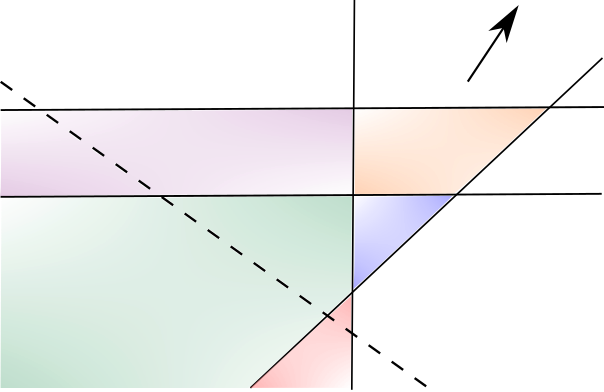}
\caption{The mass parameter $\mass$ is represented by an arrow. The compact skeleton of $\htvar_G(\FI)$ is the union of a $\mathbb{P}^2$ and of a Hirzebruch surface, meeting cleanly along a $\mathbb{P}^1$. The five simple objects in $\cO_G^{\Betti}(\FI,\mass)$ correspond to shaded components. The dotted line is the image of a fiber of the $J$-holomorphic moment map, which is symplectically a plumbing of three copies of $T^*S^3$ along circles.}\label{figure:drawing-4d}
\centering
\end{figure}

\end{example}

\bibliographystyle{plain}
\bibliography{refs}

\end{document}